\documentclass[a4paper]{article}
\usepackage{latexsym}
\usepackage{amssymb}
\usepackage{bm}
\usepackage{amsmath}
\usepackage{mathtools}
\usepackage{graphicx}
\usepackage{epsfig}
\usepackage{amsthm}
\usepackage{caption}
\usepackage{subcaption}
\usepackage{morefloats}
\usepackage{hyperref}
\hypersetup{colorlinks=true,linkcolor=blue,citecolor=blue,urlcolor=blue}
\usepackage{epstopdf}
\usepackage{xcolor}
\usepackage[top=2.5cm, left=2.0cm, right=2.0cm, bottom=2.5cm]{geometry}
\newtheorem{theorem} {\bf Theorem}
\newtheorem{lemma} {\bf Lemma}

\newtheorem{remark}{Remark}[section]

\usepackage{booktabs}
\usepackage{appendix}
\allowdisplaybreaks
\begin{document}
\title{\bf {N}itsche method for {N}avier-{S}tokes equations with slip boundary conditions: {C}onvergence analysis and VMS-LES stabilization}
\author{\textbf{Aparna Bansal, Nicol\'{a}s Alejandro Barnafi, Dwijendra Narain Pandey}
	\\ }
\date{}
 \maketitle
\begin{center}
{\bf Abstract}
\end{center}
 \noindent
In this paper, we analyze the Nitsche's method for the stationary Navier-Stokes equations on Lipschitz domains under minimal regularity assumptions. Our analysis provides a robust formulation for implementing slip (i.e. Navier) boundary conditions in arbitrarily complex boundaries. The well-posedness of the discrete problem is established using the Banach Ne\v{c}as Babu\v{s}ka and the Banach fixed point theorems under standard small data assumptions, and we also provide optimal convergence rates for the approximation error. Furthermore, we propose a VMS-LES stabilized formulation, which allows the simulation of incompressible fluids at high Reynolds numbers. We validate our theory through several numerical tests in well established benchmark problems. 
\vspace{0.5cm}
\vspace{0.5cm}
\vspace{0.5cm}
\section{Introduction}
The Navier Stokes equations describe the motion of incompressible fluids, and they pose significant challenges across different disciplines. The numerical analysis community has put considerable efforts to develop robust and efficient numerical techniques for approximating their efficient numerical approximation. It is typically assumed that the fluid adheres to the walls of its recipient, which is known as the no-slip boundary condition.  The accuracy of this assumption has been a subject of intense debate \cite{goldstein1938modern}. There are many fluid flow phenomena such as  inkjet printing \cite{MR3179782}, pipe flow \cite{berg2008two}, complex turbulent flows \cite{MR1753115}, and slide coating \cite{christodoulou1989fluid} that are better addressed using boundary conditions that allow for the fluid to slip through the walls, known also as Navier boundary conditions.

Navier boundary conditions impose a constraint only in the normal direction, which makes their implementation not trivial. To do so, the existing approaches  can be separated into (1) Lagrange Multiplier 
\cite{MR3117433,MR1707832,verfurth1986finite,MR1124131,MR2030987} and (2) Penalty methods with regularization term \cite{cremonesi2017basal,MR3171814,MR2571324}.
Both approaches weakly enforce the slip condition into the weak formulation, which although useful, can present erratic behavior known as the Babuška-type paradox, which can result in a loss of convergence \cite{MR4379970}. We highlight the stabilized formulation from \cite{MR3510584} and the non-conforming penalty formulation analyzed in \cite{MR3974685,MR4134882} which adequately characterizes the impact of the variational crimes. In general, penalization schemes avoid the Babu\v{s}ka paradox but require additional parameters.

One particular method for imposing boundary conditions weakly is Nitsche's method \cite{MR1365557}, which can be regarded as an Augmented Lagrangian formulation for imposing boundary conditions with a Lagrange multiplier. A drawback of this method, as in other penalty methods, is that it requires a stabilization constant $\gamma$ that must be sufficiently large.  By adding a stabilization term to the weak formulation of the problem, Nitsche's method addresses the issues that arise from the strong imposition of boundary conditions on approximate geometries, as well as allow for a natural formulation for non-trivial boundary conditions. More recently, a specific treatment of the Navier boundary condition has been studied in \cite{MR3759094} for the Oseen problem. We also highlight the work by Gjerde and Scott for curved boundaries in \cite{MR4379970} and on kinetic instabilities \cite{MR4476209}. The convergence analysis for a stabilized finite element formulation was also very recently developed \cite{araya:hal-04077986}.

The turbulent behavior of the Navier Stokes equations for high Reynolds numbers gives rise to numerical instabilities that make their numerical approximation very challenging and severely impact the accuracy of the finite element method approximations. These issues  can be alleviated using stabilized schemes such as the Streamline Upwind Petrov Galerkin (SUPG) method, the Galerkin Least square (GLS) method and the Variational Multiscale (VMS) method (see \cite{sendur2018comparative} for a review). In addition to the numerical instabilities, it is fundamental to approximate the unresolved scales which is done with turbulence models i.e.  Reynolds Averaged Navier Stokes (RANS) and Large Eddy Simulation (LES). However, RANS are constrained and might not be workable in some real world situations due to simplifications and assumptions in representing the complex nature of turbulent flows. RANS heavily relies on turbulence modeling, which result in an over-simplification and restricts the ability to accurately anticipate complex flow phenomena, as reported in \cite{sotiropoulos2012computational}.  In this paper, we will focus on VMS-LES due to its suitability for high Reynolds simulations \cite{hughes2000large,hughes2001large}, which allows for more computationally efficient LES approaches to the Navier-Stokes equations. 
 Our main contribution is twofold: on one hand, we provide a complete convergence analysis of the slip boundary conditions for the stationary Navier-Stokes equations under minimal regularity assumptions. On the other hand, we propose a robust computational framework for the simulation of the Navier boundary conditions in complex geometries with high Reynolds number. Several numerical tests validate our claims.

\subsection{Outline of the paper}
In Section~\ref{sec2}, we present the Navier-Stokes equations with the slip boundary condition, and derive the weak formulation of the problem, and then discuss the solvability analysis of the continuous case. In Section~\ref{sec3}, we introduce the Nitsche scheme, then derive the variational formulation and establish the well-posedness of the discrete Oseen problem using the Banach Ne\v{c}as Babu\v{s}ka. The well-posedness result is extended to the Navier-Stokes equations using the Banach's fixed point theorems in a standard way. In Section~\ref{sec4}, we derive a priori estimate, and prove the optimal convergence of the method.
 In Section~\ref{sec5}, we propose a stabilized scheme using the VMS-LES formulation to simulate fluids with high Reynolds number while taking into account the slip boundary condition. We consider the unsteady case in order to overcome the challenges associated with the stationary Navier-Stokes equations when studying numerical simulations at high Reynolds numbers.
  In Section~\ref{sec6}, we perform three numerical tests to support the theory was developed. The first one validates the theoretical results of the Nitsche scheme. The second one is a benchmark problem that demonstrates the consistency of our scheme for both steady and unsteady formulations at arbitrary Reynolds numbers. The third test shows the behavior of the fluid passing through a cylinder at high Reynolds numbers, further validating the findings from Section~\ref{sec6}.
  \subsection{Notations and preliminaries}
  Let $\Omega \subseteq \mathbb{R}^{n=2,3}$ be an open and bounded domain with a Lipschitz polygonal boundary $\Gamma$. The Sobolev spaces are denoted by ${W}^{k, p}(\Omega)$ with $k \geq 0$. They contain all ${L}^p(\Omega)$ with $p \geq 1$ functions whose distributional derivative up to order $k$ are in ${L}^p(\Omega)$. The norm and seminorm in  ${W}^{k, p}(\Omega)$ are denoted by  $\|\cdot\|_{k, p, \Omega}$ and $|\cdot|_{k, p, \Omega}$. For $p=2$, ${W}^{k, p}(\Omega)$ is denoted by ${H}^k(\Omega)$ and its corresponding norm and seminorm is denoted by $\|\cdot\|_{s, \Omega}:=\|\cdot\|_{k, 2, \Omega}$ and $\left|\cdot \left|_{k, \Omega}:=\right| \cdot\right|_{k, 2, \Omega}$, respectively. The space $L_0^2(\Omega)$ represents all $L^2$ functions with average zero condition over $\Omega$.
  The vector valued Sobolev spaces will be represented using boldface letter as $\textbf{\textit{H}}^k(\Omega)$. 
 Additionally, we will denote with ${H}_{\Gamma_a}^1(\Omega)$ the subspace of $H^1(\Omega)$ with homogeneous boundary conditions on $\Gamma_a$ (or ${H}_0^1(\Omega)$ when $\Gamma_{a}=\Gamma$ ), for which the Friedrichs-Poincar\`{e} inequality holds \cite{MR4489621} i.e.\,there exists $C_{\mathrm{FP}}>0$ which depends on $\Omega$ and $\Gamma_{a}$  such that
  \begin{align}\label{16}
  \|f\|_{1, \Omega} \leq C_{\mathrm{FP}}|f|_{1, \Omega} \quad \forall f \in {H}_{\Gamma_a}^1(\Omega) \text {. }
  \end{align}
  The H\"{o}lder inequality is given by 
  \begin{align}\label{15}
  \int_{\Omega}|f g| \leq\|f\|_{{L}^p(\Omega)}\|g\|_{{L}^q(\Omega)}, \quad \forall f \in {L}^p(\Omega), \forall g \in {L}^q(\Omega), \quad \text {with} \quad \frac{1}{p}+\frac{1}{q}=1.
  \end{align} 
  We have that the following Sobolev embedding $H^1(\Omega) \hookrightarrow L^q(\Omega)$ holds for $1 \leq q<\infty$ when $n=2$ and $1 \leq q \leq 6$ when $n=3$. In particular, there exists a constant $C_{\text {Sob }}(q, n)>0$, that depends only on the domain, such that
  \begin{align}\label{17}
  \|f \|_{0, q, \Omega} \leq C_{\mathrm{Sob}}(q, n)\|g\|_{1, \Omega}\, \text {for}\left\{\begin{array}{l}
  q \geq 1 \quad \text {if}\, n=2, \\
  q \in[1,6] \quad \text {if}\, n=3.
  \end{array}\right.
  \end{align}
    Finally, we will consider the norm of a product space $\mathrm{V}\times \Pi$ to be  $\|(\cdot, \cdot)\|= \|\cdot\|_\mathrm{V} + \|\cdot \|_\Pi$. Also, whenever an inequality holds for positive constants that do not depend on the mesh size, we will simply write $\lesssim$ or $\gtrsim$ and omit the constants.
\section{Continuous Problem}
\label{sec2}
In this section we introduce the model problem, define the weak formulation, discuss the stability properties, and finally prove the existence and uniqueness of the solution.
\subsection{The Model Problem }
The stationary incompressible Navier-Stokes equations is stated as follows
\begin{equation}\label{1}
\begin{aligned}
-\nu \Delta \mathbf{u}+ (\mathbf{u} \cdot \nabla) \mathbf{u}+\nabla p & =\mathbf{f} \quad \text {in}\, \Omega, \\
\operatorname{div} \mathbf{u} & =0 \quad \text {in}\,\Omega, \\ 
\mathbf{u} &= 0 \quad \text {on}\, \Gamma_D,\\
\int_\Omega p\,dx & =0,
\end{aligned}
\end{equation}
posed on a spatial domain  $\Omega \subseteq \mathbb{R}^n, n \in\{2,3\}$ with Lipschitz boundary $\Gamma$, where  $\mathbf{u}$ is the fluid velocity, $\nu>0$ is the viscosity, $p$ is the fluid pressure and  $\mathbf{f} $ represents the external body force on $\Omega$.
The Navier boundary condition on $\Gamma_{\text{Nav}}$ is defined as
\begin{equation}\label{2}
\begin{aligned}
\mathbf{u} \cdot \mathbf{n}&=0 \quad  \text{on}\,  \Gamma_{\text{Nav}}, \\
\nu \mathbf{n}^t D(\mathbf{u}) \boldsymbol{\tau}^k + \beta \mathbf{u} \cdot \boldsymbol{\tau}^k &=0 \quad \text {on}\,  \Gamma_{\text{Nav}}, \quad k =1,2,
\end{aligned}
\end{equation}

 where the boundary $\Gamma = \overline{\Gamma}_D \cup \overline{\Gamma}_{\text{Nav}}$ and $\Gamma_D \cap \Gamma_{\text{Nav}} = \emptyset$. The Navier boundary condition allows the fluid to slip along the boundary and requires that the tangential component of the stress vector at the boundary be proportional to the tangential velocity. The strain tensor is denoted by $D(\mathbf{u}) = \nabla \mathbf{u} + (\nabla \mathbf{u})^T$, and $\mathbf{n}$ and $\boldsymbol{\tau}^k$ are unit normal and tangent vectors to the boundary $\Gamma$. The friction coefficient $\beta$ will require some controlled behaviour for well-posedness as shown in Lemma~\ref{Z}.  The assumption of homogeneity in the boundary conditions is made for the sake of simplifying  the subsequent analysis, as the existence of lifting operators has already been established \cite{refId0}. Non-homogeneous boundary conditions are utilized in the numerical tests in Section~\ref{sec6}.

\subsection{Weak Formulation of the Continuous Problem}
 Define the spaces:
 $$
 \begin{aligned}
 \mathrm{V} &  \coloneqq \left\{\mathbf{v} \in \textbf{H}^1(\Omega): \,\mathbf{v} \cdot \mathbf{n}=0 \,\text {on}\,  \Gamma_{\text{Nav}} , \mathbf{v} = 0 \, \text{on}\, \Gamma_D \right\}, \\
 \Pi & \coloneqq {L}^2_0(\Omega).
 \end{aligned}
 $$
The standard weak formulation of \eqref{1} and \eqref{2} is given by: Find $\left(\mathbf{u}, p \right) \in \mathrm{V} \times \Pi $, such that
	\begin{align}\label{bb}
	\mathcal{A}\left[\left(\mathbf{u}, p \right) ;\left(\mathbf {v}, q\right)\right]=\mathcal{F}(\mathbf{v})  \quad \forall  \left(\mathbf{v}, q \right) \in \mathrm{V} \times \Pi,
	\end{align}
	where
	$$
	\mathcal{A}\left[\left(\mathbf{u}, p \right);\left(\mathbf {v}, q\right)\right] \coloneqq  \frac{\nu}{2}({D}(\mathbf{u}), {D}(\mathbf{v})) +(\mathbf{u} \cdot \nabla \mathbf{u}, \mathbf{v}) -(p, \nabla \cdot \mathbf{v}) -(q, \nabla \cdot \mathbf{u}) +\int_{\Gamma_{\text{Nav}}} \beta \sum_i\left(\boldsymbol{\tau}^i \cdot \mathbf{v}\right)\left(\boldsymbol{\tau}^i \cdot \mathbf{u}\right) d s,
	$$
	$$
	\mathcal{F}(\mathbf{v}) \coloneqq \langle \mathbf{f},\mathbf{v} \rangle,
	$$
    where $\langle\cdot, \cdot\rangle$ represents the duality pairing between $\mathrm{V}$ and its dual $\mathrm{V}^{\prime}$.  Assuming that the load function $\mathbf{f}$ belongs to $\mathrm{V}^{\prime}$, we can express \eqref{bb} in the following form: Find $(\mathbf{u},p) \in \mathrm{V} \times \Pi $, such that
\begin{equation}\label{9}
\begin{aligned}
\mathbf{A}(\mathbf{u}, \mathbf{v})+\mathbf{B}(\mathbf{v},p)+\mathbf{c}(\mathbf{u} ; \mathbf{u}, \mathbf{v}) & =\mathcal{F}(\mathbf{v}) & \forall \mathbf{v} \in \mathrm{V} , \\
\mathbf{B}(\mathbf{u}, q) & = 0 & \forall q \in \Pi,
\end{aligned}
\end{equation}
where 
\begin{equation}\label{10}
\begin{aligned}
\mathbf{A}(\mathbf{u}, \mathbf{v}) & \coloneqq \mathbf{a}(\mathbf{u},\mathbf{v})+\mathbf{a}^{\partial}_{\tau}(\mathbf{u}, \mathbf{v}),\\
\mathbf{B}(\mathbf{v}, p)& \coloneqq \mathbf{b}(\mathbf{v},p),
\end{aligned}
\end{equation}
with forms defined so that 
\begin{align*}
 \mathbf{a}(\mathbf{u}, \mathbf{v}) &\coloneqq \frac{\nu}{2}(D(\mathbf{u}), D(\mathbf{v})), \\ \mathbf{b}(\mathbf{v}, p) &\coloneqq  -(\operatorname{div} \mathbf{v}, p),\\
\mathbf{a}^{\partial}_{\tau}(\mathbf{u}, \mathbf{v}) &\coloneqq \int_{\Gamma_{\text{Nav}}} \beta \sum_i\left(\boldsymbol{\tau}^i \cdot \mathbf{u}\right)\left(\boldsymbol{\tau}^i \cdot \mathbf{v}\right) d s, \\ \mathbf{c}(\mathbf{w} ; \mathbf{u},\mathbf{v}) &\coloneqq  (\mathbf{w} \cdot \nabla \mathbf{u}, \mathbf{v}),\\
\mathcal{F}\left(\mathbf{v}\right) &\coloneqq  \langle \mathbf{f},\mathbf{v}\rangle.
\end{align*}
The derivation of the weak formulation is detailed in \cite{MR4379970}. In what remains of this section, we will establish the continuity, ellipticity and inf-sup properties of some operators that will be instrumental for our analysis ahead.
\begin{lemma}\label{qa10}
There exist positive constants $C_1$, $C_2$, and $C_3$ such that
\begin{equation*}
\begin{aligned}
\left|\mathbf{A}\left(\mathbf{u}, \mathbf{v}\right)\right| & \leq C_{1} \left\|\mathbf{u}\right\|_{1, \Omega}\left\|\mathbf{v}\right\|_{1, \Omega} && \forall \mathbf{u}, \mathbf{v} \in \mathrm{V}, \\
\left|\mathbf{c}\left(\mathbf{w} ; \mathbf{u}, \mathbf{v}\right)\right| & \leq C_{2} \left\|\mathbf{w} \right\|_{1, \Omega}\left\|\mathbf{u} \right\|_{1, \Omega}\left\|\mathbf{v} \right\|_{1, \Omega}&& \forall \mathbf{u}, \mathbf{v}, \mathbf{w} \in \mathrm{V}, \\
|\mathbf{B}(\mathbf{v},p)| & \leq C_{3}\|\mathbf{v}\|_{1,\Omega}\|p\|_{0,{\Omega}} && \forall  \mathbf{v} \in \mathrm{V}, p \in \Pi, \\
 |\mathcal{F}(\mathbf{v})| &\leq\left\|\mathbf{f} \right\|_{\mathrm{V}^{\prime}} \|\mathbf{v}\|_{\mathrm{V} }&& \forall  \mathbf{v} \in \mathrm{V}.
\end{aligned}
\end{equation*}	
\end{lemma}
\begin{proof}
	These inequalities are a direct consequence of \eqref{16}, \eqref{15}, \eqref{17}, Cauchy Schwarz and H\"{o}lder's inequalities. The constants $C_1$, $C_2$, and $C_3$ depend on the domain $\Omega$, $\nu$ and $\beta$.
\end{proof}
Let $\mathrm{Z}$ be the kernel of $\mathbf{B}$, that is
 $$
 \mathrm{Z} \coloneqq \left\{\mathbf{v} \in \mathrm{V} : \mathbf{B}(\mathbf{v},p)=0, \forall p \in \Pi \right\}=\left\{\mathbf{v} \in \mathrm{V}:(\operatorname{div} \mathbf{v},p)=0, \forall p \in \Pi \right\}.
 $$
It can be rewritten as:
 $$
 \mathrm{Z} \coloneqq \left\{\mathbf{v} \in \mathrm{V}: \operatorname{div} \mathbf{v}= 0 \text {in}\, \Omega\right\}.
 $$
 \begin{lemma}\label{anj}
   There exists a minimum positive constant $C_b$ for sufficiently small values of $|b|$ (when $b<0$). Due to the boundary condition on $\Gamma_D = \Gamma \setminus \Gamma_{\text{Nav}}$, there is a  Poincar\`{e} inequality \cite{scott2018introduction} of the form
     \begin{equation}\label{eq:eigen}
  \int_{\Omega}|\mathbf{v}|^2 d \mathbf{x} \leq C_b\left(\frac{1}{2} \int_{\Omega}|{D}(\mathbf{v})|^2 d \mathbf{x}+\int_{\Gamma_{\text{Nav}}} b \left|P_T \mathbf{v}\right|^2 d s\right) \qquad\forall \mathbf{v} \in \mathrm{Z}.
   \end{equation}
    The constant $C_{b}$ depends on $\Omega $, $\Gamma_{\text{Nav}}$, and $b$. It is considered to be the smallest positive constant such that \eqref{eq:eigen} holds. For more details on this condition, see \cite{scott2018introduction}.
\end{lemma}
 The following lemma establishes the ellipticity of $\mathbf{A}$ on $\mathrm{Z}$.
 \begin{remark}
 We highlight that this result allows for negative values of
$\beta$. 
\end{remark}
  \begin{lemma}\label{Z}
  	There exist a constant $ \xi>0$ depends on $\beta$, $\nu$, $\Omega$ and $\Gamma_{\text{Nav}}$ such that
  	\begin{align}\label{19}
  	\mathbf{A}(\mathbf{v}, \mathbf{v}) \geq \xi \|\mathbf{v}\|_{1,\Omega} \quad \forall \mathbf{v} \in \mathrm{Z},
  	\end{align}
  	where $\xi = \frac{\nu}{C_{\beta/\nu}}$.
  \end{lemma}
  \begin{proof} Consider the bilinear form 
  	$$\mathbf{A}(\mathbf{v},\mathbf{v}) = \frac{\nu}{2}({D}(\mathbf{v}), {D}(\mathbf{v}))+\int_{\Gamma_{\text{Nav}}} \beta \sum_i\left(\boldsymbol{\tau}^i \cdot \mathbf{v}\right)\left(\boldsymbol{\tau}^i \cdot \mathbf{v}\right) d s.$$
  By introducing the tangent space $T$ and the projection $P_T$ onto the tangent space, allows the boundary term to be expressed in coordinate free form i.e.  $P_T = I - \mathbf{n}  \otimes \mathbf{n}$ \cite{MR3759094}. Then
  	$$
  	\sum_{i=1}^{d-1}\left(\boldsymbol{\tau}^i \cdot \mathbf{v}\right)\left(\boldsymbol{\tau}^i \cdot \mathbf{u}\right)=\left(P_T \mathbf{v}\right) \cdot\left(P_T \mathbf{u}\right),
  	$$
  and also we use the Lemma~\ref{anj} as follows:
  	\begin{align*}
  	\mathbf{A}(\mathbf{v},\mathbf{v}) &= \frac{\nu}{2} \int_{\Omega}|D( \mathbf{v})|^2 d \mathbf{x}+\int_{\Gamma_{\text{Nav}}} \beta\left|P_T \mathbf{v}\right|^2 d s \\&
  	 = \frac{\nu}{4} \int_{\Omega}|D( \mathbf{v})|^2 d \mathbf{x} + 
 \frac{\nu}{4} \left( \int_{\Omega}|{D}(\mathbf{v})|^2 d \mathbf{x}+\int_{\Gamma_{\text{Nav}}} \frac{4 \beta}{\nu} \left|P_T \mathbf{v}\right|^2 d s\right) \\&
  	 \geq \frac{\nu}{4} \int_{\Omega}|D( \mathbf{v})|^2 d \mathbf{x} + \frac{\nu}{4 C_{4 \beta/\nu}} \int_{\Omega}|\mathbf{v}|^2 d \mathbf{x} \\&
  	 \geq \xi \|\mathbf{v}\|_{1,\Omega} \quad \forall \mathbf{v} \in \mathrm{Z}.
  	\end{align*}
where  $\xi = \min \{\frac{\nu}{2}, \frac{\nu}{ 4C_{4 \beta/\nu }}\}$, and $C_{4 \beta/\nu}$ is the constant denoted  as $C_b$ with $b=4 \beta/\nu$. 
    The constant $C_b$ depends on $\Omega$, $\Gamma_{\text{Nav}}$ and $b$, but given that the geometry is fixed, we denote the dependence only on b. Interestingly, it was
shown in \cite{MR4476209} that the function $b \mapsto C_b$ is monotone decreasing.
  
  \end{proof}	
Next, we present the continuous inf-sup condition for the bilinear form $\mathbf{B}$.
  	\begin{lemma}\label{b}
  		There exist a positive constant $\theta >$0 dependent on the shape of the domain $\Omega$ such that
  		\begin{align}\label{20}
  		\sup _{\mathbf{0} \neq \mathbf{v} \in \mathrm{V}} \frac{\left| \mathbf{B}( \mathbf{v},q)\right | }{\|\mathbf{v}\|_{1,\Omega}} \geq \theta \| q \|_{0,\Omega} \quad \forall q \in \Pi.
  		\end{align}
  	\end{lemma}
  	\begin{proof}
  		See \cite{MR548867} for a proof.
  	\end{proof}
  
  \subsubsection{The fixed point operator}
  In this section, we make the assumption that the data is sufficiently small and utilize the Banach fixed point theorem to establish the existence and uniqueness of a solution for equation \eqref{9}. 
  Let us introduce the bounded set
  \begin{align}\label{11}
  \mathcal{K} \coloneqq \left\{\mathbf{v} \in \mathrm{V}:\|\mathbf{v}\|_{1,\Omega} \leq {\alpha}^{-1}  \|\mathbf{f}\|_{\mathrm{V}^{\prime}} \right\},
  \end{align}
  with $\alpha$ is a positive constant defined in \eqref{29}. Now, we define the fixed point operator as
  \begin{align}\label{12}
  \mathcal{J}: \mathcal{K} \rightarrow \mathcal{K}, \quad \mathbf{w} \rightarrow \mathcal{J}(\mathbf{w})=\mathbf{u},
  \end{align}
  where given $\mathbf{w} \in \mathcal{K}$, $\mathbf{u}$ is the first component of the solution of the linearized version of problem \eqref{9}: Find $(\mathbf{u},p) \in \mathbf{V} \times \Pi $, such that
  \begin{equation}\label{13}
  \begin{array}{rlrl}
  \mathbf{A}(\mathbf{u}, \mathbf{v})+\mathbf{B}(\mathbf{v},p)+\mathbf{c}(\mathbf{w} ; \mathbf{u}, \mathbf{v}) & =\mathcal{F}(\mathbf{v}) & \forall \mathbf{v} \in \mathrm{V} , \\
  \mathbf{B}(\mathbf{u}, q) & = 0 & \forall q \in \Pi.
  \end{array}
  \end{equation}
 Based on the above, we establish the following relation
  \begin{align}\label{14}
  \mathcal{J}(\mathbf{u})=\mathbf{u} \Leftrightarrow(\mathbf{u},p) \in  \mathrm{V} \times \Pi  \,\text{satisfies \eqref{9}}.
  \end{align}
  Therefore, to demonstrate the well-posedness of \eqref{9}, it is sufficient to prove that the fixed point operator $\mathcal{J}$ possesses a unique fixed point.
Let us now introduce the bilinear form as
  	\begin{align}\label{27}
  	\mathcal{C}\left[(\mathbf{u},p);( \mathbf{v},q)\right]=\mathbf{A}(\mathbf{u},\mathbf{v})+\mathbf{B}(\mathbf{u},q)+\mathbf{B}(\mathbf{v},p).
  	\end{align}
   \begin{lemma}
       There exist a positive constant $\alpha$ such that 
       \begin{align}\label{28}
  	\sup _{\mathbf{0} \neq(\mathbf{u},p) \in \mathrm{V} \times \Pi} \frac{\mathcal{C}\left[(\mathbf{u},p);(\mathbf{v},q)\right]}{\|\mathbf{v},q)\|} \geq \alpha \|( \mathbf{u},p)\| \quad \forall(\mathbf{u},p) \in \mathrm{V} \times \Pi,
  	\end{align}
   with 
  	\begin{align}\label{29}
  	\alpha = \frac{\xi \theta}{2 \xi + {\theta} +1},
  	\end{align}
   \end{lemma} 
   where $\xi$ and $\theta$ are the coercivity and inf-sup stability constants respectively.
   \begin{proof}
 \cite[Proposition 2.36]{MR4242224} in which Lemma~\ref{Z} and Lemma~\ref{b} are used.
   \end{proof}
  Now, it is possible to establish the well-posedness of $\mathcal{J}$ and thus the unisolvence of \eqref{9}. These results are well-established, so we report them without proofs.
  		\begin{theorem}\label{a}
  			Assume that
  			\begin{align}\label{30}
  			\frac{2}{\alpha^2} \|\mathbf{f}\|_{\mathrm{V}^{\prime}} \leq 1.
  			\end{align}
  			Then, given $\mathbf{w} \in \mathbf{K}$, there exists a unique $\mathbf{u} \in \mathbf{K}$ such that $\mathcal{J}(\mathbf{w})=\mathbf{u}$.
  		\end{theorem}
  	\begin{theorem}\label{c}
  		Let $\mathbf{f} \in  \mathrm{V}^{\prime}$ such that
  		\begin{align}\label{38}
  		\frac{2}{\alpha^2}\|\mathbf{f}\|_{\mathrm{V}^{\prime}}\leq 1.
  		\end{align}
  		Then, there exists a unique $(\mathbf{u},p) \in \mathrm{V} \times \Pi $ solution to \eqref{9}. In addition, there exists $C>0$ such that
  		\begin{align}\label{39}
  		\|\mathbf{u}\|_{1}+\|p\|_{0,\Omega} \leq C\|\mathbf{f}\|_{\mathrm{V}^{\prime}}.
  		\end{align}
  	\end{theorem}
  \section{Discrete Problem}
  \label{sec3}
  This section studies the solvability and convergence analysis of the Nitsche's scheme for the problem \eqref{9}. We assume that the polygonal computational domain ${\Omega}$ is discretized using a collection of regular partitions, denoted as $\{\mathcal{T}_h\}_{h>0}$, where $\Omega \subset \mathbb{R}^n$ is divided into simplices $T$ (triangles in 2D or tetrahedra in 3D) with a diameter $h_{{T}}$. The characteristic length of the finite element mesh $\mathcal{T}_h$ is denoted as $h:=\max _{{T} \in \mathcal{T}_h} h_{{T}}$. For a given triangulation $\mathcal{T}_h$, we define $\mathcal{E}_h$ as the set of all faces in $\mathcal{T}_h$, with the following partitioning
  		$$
  		\mathcal{E}_h:=\mathcal{E}_{\Omega} \cup \mathcal{E}_D \cup \mathcal{E}_{\text{Nav}}
  		$$
  		where $\mathcal{E}_{\Omega}$ represents the faces  lying in the interior of $\Omega$, $\mathcal{E}_{\text{Nav}}$ represents the faces lying on the boundary $\Gamma_{\text{Nav}}$, and $\mathcal{E}_D$ represents the  faces lying on the boundary $\Gamma_D$. Additionally, $h_e$ denotes the $(n-1)$ dimensional diameter of an  face. Here \emph{faces} loosely refer to the geometrical entities of co-dimension 1.
  	    Now, let us introduce the finite element pair.
  		$$
  		\begin{aligned}
  		& \mathrm{V}_h \coloneqq \left\{\mathbf{v}_h \in \mathbf{C}(\overline{\Omega}) : \mathbf{v}_h =0 \text{on}\, E \in \mathcal{E}_D, \left.\mathbf{v}_h\right|_{T} \in \mathbf{P}_k({T}) \quad \forall {T} \in \mathcal{T}_h\right\}, \\
  		& \mathrm{\Pi}_h \coloneqq \left\{q_h \in \mathrm{C}(\overline{\Omega}) :\left.q_h\right|_{{T}} \in \mathrm{P}_{k-1}({T}) \quad \forall {T} \in \mathcal{T}_h\right\} \cap \Pi ,
  		\end{aligned}
  		$$
  		where $\mathrm{P}_k(T)$ is the space of polynomials of degree less than or equal to $k$ defined on $T$.
  	\begin{remark}
  	It is noted that $\mathrm{\Pi}_h$ is a subspace of $\mathrm{\Pi}$, but $\mathrm{V}_h$ is not a subspace of $\mathrm{V}$. In that sense, Nitsche's method can be considered a non-conforming finite element approximation.
  	 \end{remark}
  		\subsection{Nitsche's Method}
  	The main objective of Nitsche's method \cite{MR1365557,MR3759094} is to impose boundary conditions weakly so that they hold only asymptotically, which for this work will hold for the Navier boundary condition $\mathbf{u} \cdot \mathbf{n} = 0$ only. As a result, the weak formulation with the Nitsche method can be expressed as follows:
  		Find $\left(\mathbf{u}_h, p_h \right) \in \mathrm{V}_h \times \Pi_h $, such that
  		\begin{align}\label{41}
  		\mathcal{A}_h\left[\left(\mathbf{u}_h, p_h\right) ;\left(\mathbf{v}_h, q_h\right)\right]=\mathcal{F}(\mathbf{v}_h)  \quad \forall  \left(\mathbf{v}_h, q_h \right) \in \mathrm{V}_h \times \Pi_h,
  		\end{align}
  with the forms are defined as 
  		$$
  		\begin{aligned}
  		\mathcal{A}_h\left[(\mathbf{u}_h, p_h);(\mathbf{v}_h, q_h)\right]  &\coloneqq \sum_{T\in \mathcal{T}_{h}}\bigg(\frac{\nu}{2}(D(\mathbf{u}_h), D(\mathbf{v}_h))+(\mathbf{u}_h \cdot \nabla \mathbf{u}_h, \mathbf{v}_h)-(p_h, \nabla \cdot \mathbf{v}_h)-(q_h, \nabla \cdot \mathbf{u}_h)\bigg) \\& +
   		\sum_{E\in \mathcal{E}_{\text{Nav}}}\bigg(-\int_{E} \mathbf{n}^t(\nu D(\mathbf{u}_h)-p_h I) \mathbf{n}(\mathbf{n} \cdot \mathbf{v}_h) d s-\int_{E} \mathbf{n}^t(\nu {D}(\mathbf{v}_h)-q_h I) \mathbf{n}(\mathbf{n} \cdot \mathbf{u}_h) d s  \\& + \int_{E} \beta \sum_i\left(\boldsymbol{\tau}^i \cdot \mathbf{v}_h\right)\left(\boldsymbol{\tau}^i \cdot \mathbf{u}_h\right) d s 
  		+\gamma \int_{E} {h_e}^{-1}(\mathbf{u}_h \cdot \mathbf{n})(\mathbf{v}_h \cdot \mathbf{n}) d s \bigg), 
  		\end{aligned}
  		$$
  		
  		$$
    \begin{aligned}
  		\mathcal{F}(\mathbf{v}_h) &\coloneqq \langle \mathbf{f},\mathbf{v}_h\rangle,
    \end{aligned}
  		$$
  	and $\gamma>0$ is a positive constant that needs to be chosen sufficiently large, as proved in Lemma~\ref{S} later. We can rewrite \eqref{41} as:
  	    Find $(\mathbf{u}_h,p_h) \in \mathrm{V}_h \times \mathrm{\Pi}_h $, such that
  	    \begin{equation}\label{E}
  	    \begin{aligned}
  	    \begin{array}{rlrl}
  	    \mathbf{A}_h(\mathbf{u}_h, \mathbf{v}_h)+\mathbf{B}_h(\mathbf{v}_h,p_h)+\mathbf{c}(\mathbf{u}_h ; \mathbf{u}_h, \mathbf{v}_h) & =\mathcal{F}(\mathbf{v}_h) & \forall \mathbf{v}_h \in \mathrm{V}_h, \\
  	    \mathbf{B}_h(\mathbf{u}_h, q_h) & = 0 & \forall q_h \in \mathrm{\Pi}_h,
  	    \end{array}
  	    \end{aligned}
  	    \end{equation}
 where \begin{equation}\label{F}
  	    \begin{aligned}
  	    \mathbf{A}_h(\mathbf{u}_h, \mathbf{v}_h) & \coloneqq  \mathbf{a}(\mathbf{u}_h,\mathbf{v}_h)+\mathbf{a}^{\partial}_{\tau}(\mathbf{u}_h, \mathbf{v}_h)+\mathbf{a}^{\partial}_{\gamma}(\mathbf{u}_h, \mathbf{v}_h)-\mathbf{a}^{\partial}_{c}(\mathbf{u}_h, \mathbf{v}_h)-\mathbf{a}^{\partial}_{c}(\mathbf{v}_h, \mathbf{u}_h),\\
  	    \mathbf{B}_h(\mathbf{u}_h, q_h) &\coloneqq \mathbf{b}(\mathbf{u}_h,q_h)+\mathbf{b}^{\partial}( \mathbf{u}_h,q_h),
  	    \end{aligned}
  	    \end{equation}
  	    with forms defined so that 
  	    \begin{align}\label{mn}
  	    \mathbf{a}(\mathbf{u}_h, \mathbf{v}_h) &\coloneqq \sum_{T\in \mathcal{T}_{h}} \frac{\nu}{2}(D(\mathbf{u}_h), D(\mathbf{v}_h)), \nonumber \\ \mathbf{b}(\mathbf{u}_h, q_h) &\coloneqq - \sum_{T\in \mathcal{T}_{h}} (\operatorname{div} \mathbf{u}_h, q_h), \nonumber \\
  \mathbf{b}^{\partial}( \mathbf{u}_h,q_h) &\coloneqq \sum_{E\in \mathcal{E}_{\text{Nav}} } \int_{E}q_h(\mathbf{n} \cdot \mathbf{u}_h) d s, \nonumber \\
  \mathbf{a}^{\partial}_{c}(\mathbf{u}_h, \mathbf{v}_h) &\coloneqq \sum_{E\in \mathcal{E}_{\text{Nav}}} \int_{E} \mathbf{n}^t \nu D(\mathbf{u}_h) \mathbf{n}(\mathbf{n} \cdot \mathbf{v}_h) d s,  \nonumber \\
  \mathbf{a}^{\partial}_{\tau}(\boldsymbol{u}_h, \mathbf{v}_h) &\coloneqq \sum_{E\in \mathcal{E}_{\text{Nav}}} \int_{E} \beta \sum_i\left(\boldsymbol{\tau}^i \cdot \mathbf{u}_h\right)\left(\boldsymbol{\tau}^i \cdot \mathbf{v}_h\right) d s, \nonumber \\ \mathbf{a}^{\partial}_{\gamma}(\mathbf{u}_h, \mathbf{v}_h) &\coloneqq \sum_{E\in \mathcal{E}_{\text{Nav}}}\int_{E} \frac{\gamma}{h_e}(\mathbf{u}_h \cdot \mathbf{n})(\mathbf{v}_h \cdot \mathbf{n}) d s, \nonumber \\
\mathcal{F}\left(\mathbf{v}_h\right) &\coloneqq \langle \mathbf{f},\mathbf{v}_h\rangle, \nonumber \\
  \mathbf{c}(\mathbf{w}_h ; \mathbf{u}_h,\mathbf{v}_h) &\coloneqq \sum_{T\in \mathcal{T}_{h}} (\mathbf{w}_h \cdot \nabla \mathbf{u}_h, \mathbf{v}_h).
 \end{align}
  		\subsubsection{Discrete stability properties}
  			In this section, we leverage well-known inverse and trace inequalities to establish two results: the ellipticity of $\mathbf{A}_h$ and the inf-sup property $\mathbf{B}_h$.
  			\begin{lemma}\label{qa1}
  				Let $\mathbf{v}_h \in \mathrm{V}_h$ then for each $T \in \mathcal{T}_h ;\, l, m \in \mathbb{N}$, with $0 \leq m \leq l$, there exists a positive constant $C_{4}$, independent of $T$, such that
  				$$
  				\left|\mathbf{v}_h\right|_{l, T} \leq C_{4} h_T^{m-l}\left|\mathbf{v}_h\right|_{m, T}.
  				$$
  			\end{lemma}
  			\begin{proof}
	          \cite[Lemma 12.1]{MR4242224}.
  			\end{proof}
  			\begin{lemma}\label{qa2}
  				 Let $\mathbf{v}_h \in \mathrm{V}_h$ then for each $T \in \mathcal{T}_h, E \subset \partial T$, there exists a positive constant $C_{5}$, independent of $T$, such that
  				$$
  				\left\|\mathbf{v}_h\right\|_{0, E} \leq C_{5} h_T^{-\frac{1}{2}}\left\|\mathbf{v}_h\right\|_{0, T}.
  				$$
  			\end{lemma}
  			\begin{proof}
  	\cite{warburton2003constants}.
  			\end{proof}
 We need to define the energy norm on $\mathrm{V}_h$ as
\begin{align}\label{a1}
  		\| \mathbf{v}_h \|_{1,h}^2 \coloneqq \| \nabla \mathbf{v}_h \|_{0,\Omega}^2   + \sum_{E \in \mathcal{E}_{\text{Nav}}} \frac{1}{h_e} \| \mathbf{v}_h \cdot \mathbf{n} \|_{0,E}^2.
  		\end{align}
  		\noindent
We highlight that this norm, in contrast to the one found in \cite{araya:hal-04077986}, contains only the boundary term on the normal component. This is happens because we consider only the Navier boundary condition weakly, but indeed our proof would remain mostly unmodified if we were to consider the Dirichlet boundary conditions weakly as well. We would only require adding the missing norm for the velocity in this energy norm. All estimates would remain the same. We now state the continuity of the discrete bilinear forms in terms of this norm.
  		\begin{theorem}\label{qa4}
  There exist positive constants $C_7$, $C_8$, and $C_9$ independent of $h$ such that
  		\begin{equation*}
  		\begin{aligned}
  		\left|\mathbf{A}_h(\mathbf{u}_h, \mathbf{v}_h)\right| \leq& C_7 \|\mathbf{u}_h\|_{1, h}\|\mathbf{v}_h\|_{1,h}, &&  \forall \mathbf{u}_h, \mathbf{v}_h \in \mathrm{V}_h,
  		\\
  		\left|\mathbf{c}\left(\mathbf{w}_h; \mathbf{u}_h, \mathbf{v}_h\right)\right| \leq& {C}_8 \left\|\mathbf{w}_h\right\|_{1,h}\|\mathbf{u}\|_{1,h}\|\mathbf{v}_h\|_{1, h}, && \forall \mathbf{w}_h, \mathbf{u}_h, \mathbf{v}_h \in \mathrm{V}_h,
  		\\
  		\left|\mathbf{B}_h(\mathbf{v}_h, q_h)\right| \leq& C_9 \|\mathbf{v}_h\|_{1, h}\|q_h\|_{0, \Omega}, && \forall \mathbf{v}_h \in \mathrm{V}_h, q_h \in \Pi_h.
  		\end{aligned}
  		\end{equation*}
        \end{theorem}  
  	\begin{proof}	
  	  The proof of above inequalities is a direct consequence of Lemmas~\ref{qa1},~\ref{qa2}, the Sobolev embedding with $r=4$, and Cauchy-Schwarz and H\"{o}lder inequalities. Moreover, it can be seen that  
    \begin{align*}
     C_7 &\sim \frac{\nu}{2} + \beta +\gamma + \nu C_{5} \\
     C_8 &\sim C_{\text{Sob}}(4,n) \\
     C_9 &\sim 1+C_5
    \end{align*}
     i.e., above constants can bounded by a constant that depends only on the trace inequality constant $C_5$, and on
the parameters $\nu$ ,$\gamma$ and $\beta$.
  	\end{proof}
The discrete kernel of $\mathbf{B}_h$ is defined by:
  		$$
  		\mathrm{Z}_h \coloneqq \left\{\mathbf{v}_h \in \mathrm{V}_h: \mathbf{B}_h\left( \mathbf{v}_h,p_h \right)=0, \forall {p}_h \in \mathrm{\Pi}_h\right\}.
  		$$
  		It can be equivalently written as
  		\begin{align}\label{a7}
  		\mathrm{Z}_h \coloneqq \left\{\mathbf{v}_h \in \mathrm{V}_h: \sum_{T\in \mathcal{T}_{h}} \int_{T} p_h \nabla \cdot \mathbf{v}_h d \mathbf{x} - \sum_{E\in \mathcal{E}_{\text{Nav}}} \int_{E} {p_h} (\mathbf{n} \cdot \mathbf{v}_h) ds =0 \quad \forall p_h \in \Pi_h \right\}.
  		\end{align}
  		The following lemma establishes the ellipticity of $\mathbf{A}$ on $\mathrm{V}_h$.
  		\begin{lemma}\label{S}
  			 There exist positive constants $\gamma_0, C_0$ and $C_S=C_S\left(\beta, \nu \right)$, independent on $h$, such that
  			\begin{align}\label{G}
  			\mathbf{A}_h(\mathbf{v}_h, \mathbf{v}_h) \geq C_S \|\mathbf{v}_h\|_{1,h} \quad \forall \mathbf{v}_h \in \mathrm{Z}_h,
  			\end{align}
  			where $C_S = \min \{\xi - {\nu C_{5}^2}{ C_0},\gamma -\frac{\nu}{C_0}\}$ with $\gamma \geq \gamma_0 > \frac{\nu}{C_0}$, $C_0 < \frac{\xi}{C_{5}^2 \nu }$, and $\xi$ is the coercivity constant defined in the continuous case.
  		\end{lemma}
  		\begin{proof}
  		For the proof, we use \eqref{F} to obtain 
  			$$
  			\begin{aligned}
  			\mathbf{A}_h(\mathbf{v}_h, \mathbf{v}_h) &=\frac{\nu}{2}(D(\mathbf{v}_h), D(\mathbf{v}_h)) + \sum_{E \in \mathcal{E}_{\text{Nav}}} \int_{E} \beta \sum_i\left(\boldsymbol{\tau}^i \cdot \mathbf{v}_h\right)\left(\boldsymbol{\tau}^i \cdot \mathbf{v}_h\right) d s  \\ & - \sum_{E \in \mathcal{E}_{\text{Nav}}}  2 \int_{E} \mathbf{n}^t \nu D(\mathbf{v}_h) \mathbf{n}(\mathbf{n} \cdot \mathbf{v}_h) ds + \frac{\gamma}{h_e} \int_{E} (\mathbf{v}_h \cdot \mathbf{n})(\mathbf{v}_h \cdot \mathbf{n}) ds.
  			\end{aligned}
  			$$
  			Using Lemma~\ref{Z}, Lemma~\ref{qa2}, Cauchy Schwarz and H\"{o}lder's inequalities, the following estimate can be established 
  			$$
  			\begin{aligned}
  			\mathbf{A}_h(\mathbf{v}_h, \mathbf{v}_h) &\geq \xi \|\nabla \mathbf{v}_h\|^2_{0,\Omega} -2 \nu \sum_{E \in \mathcal{E}_{\text{Nav}}} \|\mathbf{v}_h \cdot \mathbf{n} \|_{0,E} \| D \mathbf{v}_h \mathbf{n} \|_{0,E} +  \frac{\gamma}{h_e} \sum_{E \in \mathcal{E}_{\text{Nav}}} \|\mathbf{v}_h \cdot \mathbf{n} \|_{0,E}^2 
  			 \\  &\geq \xi \|\nabla \mathbf{v}_h\|^2_{0,\Omega} -2 \nu \sum_{E \in \mathcal{E}_{\text{Nav}}} h_e^{-1/2} \|\mathbf{v}_h \cdot \mathbf{n} \|_{0,E} h_e^{1/2} \| D \mathbf{v}_h \mathbf{n} \|_{0,E} +  \frac{\gamma}{h_e} \sum_{E \in \mathcal{E}_{\text{Nav}}} \|\mathbf{v}_h \cdot \mathbf{n} \|_{0,E}^2  \\ &
  			\geq \xi \|\nabla \mathbf{v}_h\|^2_{0,\Omega} - 2 \nu   \sum_{E \in \mathcal{E}_{\text{Nav}}} \left( \frac{h_e C_0}{2} \| D \mathbf{v}_h  \|_{0,E}^2   +   \frac{h_e^{-1}}{2 C_0}  \|\mathbf{v}_h \cdot \mathbf{n} \|_{0,E}^2 \right)+  \frac{\gamma}{h_e} \sum_{E \in \mathcal{E}_{\text{Nav}}} \|\mathbf{v}_h \cdot \mathbf{n} \|_{0,E}^2  \\ &
  			\geq \xi \|\nabla \mathbf{v}_h\|^2_{0,\Omega} -  \nu {C_{5}^2}{C_0} \sum_{K \in \mathcal{T}_h} \| D \mathbf{v}_h  \|_{0,K}^2 -     
  			 \frac{\nu}{C_0 \gamma} \sum_{E \in \mathcal{E}_{\text{Nav}}} \frac{\gamma}{h_e}	\|\mathbf{v}_h \cdot \mathbf{n} \|_{0,E}^2  +  \frac{\gamma}{h_e} \sum_{E \in \mathcal{E}_{\text{Nav}}} \|\mathbf{v}_h \cdot \mathbf{n} \|_{0,E}^2 
  	         \\& \geq \xi \|\nabla \mathbf{v}_h\|^2_{0,\Omega} - {\nu C_{5}^2}{C_0} \| \nabla \mathbf{v}_h \|^2_{0,\Omega} +  \left(1-\frac{\nu}{C_0 \gamma}\right) \sum_{E \in \mathcal{E}_{\text{Nav}}} \frac{\gamma}{h_e} \|\mathbf{v}_h \cdot \mathbf{n} \|_{0}^2 
  	         \\&  \geq \left(\xi - {\nu C_{5}^2}{C_0}\right) \|\nabla \mathbf{v}_h\|^2_{0,\Omega} + \left(\gamma-\frac{\nu}{C_0}\right) \sum_{E \in \mathcal{E}_{\text{Nav}}} \frac{1}{h_e} \|\mathbf{v}_h \cdot \mathbf{n} \|_{0,E}^2
  	        \\& \geq C_S \|\mathbf{v}_h\|_{1,h}.
  			\end{aligned}
  			$$
  			By selecting the positive parameter $C_0$ that satisfies $C_0 < \frac{\xi}{C_{5}^2 \nu }$, we ensure that $\left(\xi - {\nu C_{5}^2}{ C_0}\right) >0 $. Additionally, we define $C_S = \min \{\xi - {\nu C_{5}^2}{ C_0},\gamma -\frac{\nu}{C_0}\}$ with $\gamma \geq \gamma_0 > \frac{\nu}{C_0}$.
  		\end{proof}
  		Next, we proceed to derive the discrete version of Lemma~\ref{b}.
    \begin{remark}
 The inf-sup condition is associated with the construction of inf-sup stable elements in
incompressible flow modeling. However, this condition is not automatically fulfilled and needs to be verified for specific choices of the approximation spaces $\mathrm{V}_h$ and $\Pi_h$. It is worth mentioning that the Taylor-Hood elements are inf-sup stable, meaning they meet the necessary condition for stability. In fact, the proof can be demonstrated using any pair of inf-sup stable elements.
    \end{remark}
  			\begin{lemma}\label{infsup}
  				There exists $\hat{\theta}>0$ independent of $h$ such that
  				\begin{align}\label{46}
  				\sup _{\mathbf{0} \neq \mathbf{v}_h \in \mathrm{V}_h} \frac{\left | \mathbf{B}_h( \mathbf{v}_h,q_h) \right | }{\|\mathbf{v}_h\|_{1,h}} \geq \hat{\theta} \| q_h \|_{0,\Omega} \quad \forall q \in \Pi_h.
  				\end{align}
  			\end{lemma}	
  \begin{proof}
 Consider the Taylor-Hood Finite element spaces $\mathrm{V}_h \times \Pi_h$ such that the
discrete inf-sup holds for the bilinear form $\mathbf{b}(\mathbf{v}_h,q_h)$, see \cite[Section 5.5]{MR3235759} i.e.\, there exist a positive constant $\hat{\theta}>$ 0, independent of h such that  
\begin{align*}
  		\sup _{\mathbf{0} \neq \mathbf{v}_h \in \mathrm{V}_{h}} \frac{\left| \mathbf{b}( \mathbf{v}_h,q_h)\right | }{\|\mathbf{v}_h\|_{1,\Omega}} \geq  \hat{\theta} \| q_h \|_{0,\Omega} \quad \forall q_h \in \Pi_h.
\end{align*}
Consider the discrete space of strongly imposed Navier conditions
$$\mathrm{V}_{h,0} = \{ \mathbf{v}_h \in \mathrm{V}_h : \mathbf{v}_h \cdot \mathbf{n} = 0 \,\text{on}\, \Gamma_{\text{nav}}\}.$$
This space naturally yields that $b^\partial(\mathbf{v}_h, q_h) = 0$ for all $\mathbf{v}_h$ in $\mathrm{V}_{h,0}$, so we obtain
\begin{align*}
  		\sup _{\mathbf{0} \neq \mathbf{v}_h \in \mathrm{V}_{h,0}} \frac{\left| \mathbf{B}_h( \mathbf{v}_h,q_h)\right | }{\|\mathbf{v}_h\|_{1,h}} = \sup_{\mathbf{0} \neq \mathbf{v}_h \in \mathrm{V}_{h,0}} \frac{\left| \mathbf{b}(\mathbf{v}_h, q_h)\right|}{\|\mathbf{v}_h\|_{1,h}} = \sup _{\mathbf{0} \neq \mathbf{v}_h \in \mathrm{V}_{h,0}} \frac{\left| \mathbf{b}( \mathbf{v}_h,q_h)\right | }{\|\mathbf{v}_h\|_{1,\Omega}} \geq \hat{\theta} \| q_h \|_{0,\Omega} \quad \forall q_h \in \Pi_h
\end{align*}
in virtue of the classical inf-sup condition.  Now, use this space for a lower bound i.e. 
\begin{align*}
  		\sup _{\mathbf{0} \neq \mathbf{v}_h \in \mathrm{V}_{h}} \frac{\left| \mathbf{B}_h( \mathbf{v}_h,q_h)\right | }{\|\mathbf{v}_h\|_{1,h}} \geq \sup _{\mathbf{0} \neq \mathbf{v}_h \in \mathrm{V}_{h,0}} \frac{\left| \mathbf{B}_h( \mathbf{v}_h,q_h)\right | }{\|\mathbf{v}_h\|_{1,h}} \geq \hat{\theta} \| q_h \|_{0,\Omega} \quad \forall q_h \in \Pi_h. 
\end{align*}
This concludes the proof.
\end{proof}
  		Next, We aim to establish the well-posedness of problem \eqref{E}. We will utilize a fixed-point operator linked to a linearized form of the problem and demonstrate that this operator has a  unique fixed-point. Equivalently, we can prove the well-posedness of problem \eqref{E} using the Banach fixed-point theorem. 
  		\subsubsection{The discrete fixed-point operator and its well-posedness}
  		Let us introduce the set
  		\begin{align}\label{qa5}
  		\mathbf{K}_{\mathbf{h}}=\left\{\mathbf{v}_h \in \mathrm{V}_h:\|\mathbf{v}_h\|_{1,h} \leq {\hat{\alpha}}^{-1}\| \mathbf{f}\|_{\mathrm{V}^{\prime}}\right\},
  		\end{align}
  		with $\hat{\alpha}>0$ being the constant defined below in Theorem~\ref{51}. Now, let us define the discrete fixed point operator as
  		$$
  		\mathcal{J}_h: \mathbf{K}_h \rightarrow \mathbf{K}_h, \quad \mathbf{w}_h \rightarrow \mathcal{J}_h\left(\mathbf{w}_h\right)=\mathbf{u}_h,
  		$$
  		where given $\mathbf{w}_h \in \mathbf{K}_h, 
     \mathbf{u}_h$ represents the first component of the solution of the linearized version of problem \eqref{F}: Find $\left( \mathbf{u}_h,p_h \right) \in \mathrm{V}_h \times \mathrm{\Pi}_h $ 
  		\begin{equation}\label{43}
  		\begin{array}{rlrl}
  		\mathbf{A}_h(\mathbf{u}_h, \mathbf{v}_h)+\mathbf{B}_h(\mathbf{v}_h,p_h)+\mathbf{c}(\mathbf{w}_h ; \mathbf{u}_h, \mathbf{v}_h) & =\mathcal{F}(\mathbf{v}_h) & \forall \mathbf{v}_h \in \mathrm{V}_h, \\
  		\mathbf{B}_h(\mathbf{u}_h, q_h) & = 0 & \forall q_h \in \mathrm{\Pi}_h.
  		\end{array}
  		\end{equation}
  		Based on the above, we can establish the following relation
    \begin{align}\label{44}
  		\mathcal{J}_h\left(\mathbf{u}_h\right)=\mathbf{u}_h \Leftrightarrow\left( \mathbf{u}_h,p_h\right) \in \mathrm{V}_h \times \mathrm{\Pi}_h 
  \quad\text{satisfies}\, \eqref{F}.
\end{align}
  		
  In order to guarantee the well-posedness of the discrete problem \eqref{F}, it is enough to demonstrate the existence of a unique fixed-point for $\mathcal{J}_h$ within the set $\mathbf{K}_h$. However, before delving into the analysis of solvability, we first need to establish the well-definedness of the operator $\mathcal{J}_h$.
Let us introduce the bilinear form.
  		\begin{align}\label{aa}
  		\mathcal{C}_h\left[(\mathbf{u}_h,p_h);( \mathbf{v}_h,q_h)\right]=\mathbf{A}_h(\mathbf{u}_h,\mathbf{v}_h)+\mathbf{B}_h(\mathbf{u}_h,q_h)+\mathbf{B}_h(\mathbf{v}_h,p_h).
  		\end{align}
  		\begin{theorem}\label{51}
  			There exist a positive constant $\hat{\alpha}$ such that 
  			\begin{align*}
  			\sup _{\mathbf{0} \neq (\mathbf{v}_h,q_h) \in \mathrm{V}_h \times \mathrm{\Pi}_h}\frac{\mathcal{C}_h\left[\left( \mathbf{u}_h,p_h\right);\left( \mathbf{v}_h,q_h \right)\right]}{\left\|\left( \mathbf{v}_h,q_h \right)\right\|} \geq \hat{\alpha}\left\|\left(\mathbf{u}_h,p_h\right)\right\| \quad \forall\left( \mathbf{u}_h,p_h\right) \in \mathrm{V}_h \times \mathrm{\Pi}_h ,
  			\end{align*}
  			with
  			\begin{align*}
  			\hat{\alpha} =\frac{C_S \hat{\theta}}{2 C_S + \hat{\theta} +1}.
  			\end{align*}
  			where $C_S$ and $\hat{\theta}$ are the coercivity and inf-sup stability constants. 
  		\end{theorem}
\begin{proof}
    Owing to Theorem~\ref{qa4}, it is clear that $\mathcal{C}_h\left[\cdot ; \cdot \right]$ is bounded. Moreover, from Lemma~\ref{S}, Lemma~\ref{infsup}, and \cite[Proposition 2.36]{MR4242224}, it is not difficult to see that above inf-sup condition holds.
\end{proof}

Now, we are in position to establish the well-posedness of $\mathcal{J}_{\text {h }}$.
  		\begin{theorem}\label{qa7}
  			Assume that
  			\begin{align}\label{53}
  			\frac{2}{\hat{\alpha}^2}\|\mathbf{f}\|_{\mathrm{V}^{\prime}}\leq 1,
  			\end{align}
  			Then, given $\mathbf{w}_h \in \mathbf{K}_h$, there exists a unique $\mathbf{u}_h \in \mathbf{K}_h$ such that $\mathcal{J}_h\left(\mathbf{w}_h\right)=\mathbf{u}_h$.	
  		\end{theorem} 
  		\begin{proof}
  			Given $\mathbf{w}_h \in \mathbf{K}_h$, we begin by defining the bilinear form:
  			\begin{align}\label{31}
  			\mathcal{A}_{\mathbf{w}_h}\left[(\mathbf{u}_h,p_h);(\mathbf{v}_h,q_h)\right]:=\mathcal{C}_h\left[( \mathbf{u}_h,p_h);( \mathbf{v}_h,q_h)\right]+\mathbf{c}(\mathbf{w}_h ; \mathbf{u}_h, \mathbf{v}_h),
  			\end{align}
  			where $\mathcal{C}_h$ and $\mathbf{c}$ are the forms defined in Theorem~\ref{51} and \eqref{mn}, respectively, that is
  			$$
  			\mathcal{A}_{\mathbf{w}_h}\left[(\mathbf{u}_h,p_h);(\mathbf{v}_h,q_h)\right]=\mathbf{A}_h(\mathbf{u}_h, \mathbf{v}_h)+\mathbf{B}_h(\mathbf{u}_h,q_h)+\mathbf{B}_h(\mathbf{v}_h,q_h)+\mathbf{c}(\mathbf{w}_h ; \mathbf{u}_h, \mathbf{v}_h).
  			$$
  			Then, problem \eqref{F} can be rewritten equivalently as: Find $(\mathbf{u}_h,p_h) \in \mathrm{V}_h \times \Pi_h$, such that
  			\begin{align}\label{32}
  			\mathcal{A}_{\mathbf{w}_h}\left[(\mathbf{u}_h,p_h);(\mathbf{v}_h,q_h)\right]=\mathcal{F}(\mathbf{v}_h) \quad \forall( \mathbf{v}_h,q_h) \in \mathrm{V}_h \times \Pi_h.
  			\end{align}
  			First, we establish the well-posedness of $\mathcal{J}$ in order to demonstrate the well posedness of problem \eqref{32} using the Banach Ne\v{c}as Babu\v{s}ka theorem  \cite[Theorem 2.6] {MR4242224}.
  		     Consider $(\mathbf{u}_h,p_h),( \hat{\mathbf{v}}_h,\hat{q}_h) \in \mathrm{V}_h \times \Pi_h$ with $(\hat{\mathbf{v}}_h,\hat{q}_h) \neq$ 0, from Theorem~\ref{qa4} we can observe that
  			$$
  			\begin{aligned}
  			\sup _{\mathbf{0} \neq(\mathbf{v}_h,p_h) \in \mathrm{V}_h \times \Pi_h} \frac{\mathcal{A}_{\mathbf{w}_h}\left[(\mathbf{u}_h,p_h);(\mathbf{v}_h,q_h)\right]}{\|(\mathbf{v}_h,q_h)\|} & \geq \frac{|\mathcal{C}_h\left[(\mathbf{u}_h,p_h);( \hat{\mathbf{v}_h},\hat{q}_h)\right]|}{\|( \hat{\mathbf{v}_h},\hat{q}_h)\|}-\frac{|\mathbf{c}(\mathbf{w}_h ; \mathbf{u}_h, \hat{\mathbf{v}}_h)|}{\|(\hat{\mathbf{v}}_h,\hat{q}_h)\|} \\
  			& \geq \frac{|\mathcal{C}_h\left[(\mathbf{u}_h,p_h);( \hat{\mathbf{v}_h},\hat{q}_h)\right]|}{\|( \hat{\mathbf{v}_h},\hat{q}_h)\|}-\|\mathbf{w}_h\|_{1,h}\|(\mathbf{u}_h,p_h)\| ,
  			\end{aligned}
  			$$
  			which together with Theorem~\ref{51} and the fact that $(\hat{\mathbf{v}}_h, \hat{p}_h)$ is arbitrary, implies
  			\begin{align}\label{33}
  			\sup _{\mathbf{0} \neq (\mathbf{v}_h,p_h) \in \mathrm{V}_h \times \Pi_h} \frac{\mathcal{A}_{\mathbf{w}_h}\left[(\mathbf{u}_h,p_h);(\mathbf{v}_h,q_h)\right]}{\|(\mathbf{v}_h,p_h)\|} \geq\left(\hat{\alpha}-\|\mathbf{w}_h\|_{1,h}\right)\|( \mathbf{u}_h,p_h)\|.
  			\end{align}
  			Hence, from the definition of set $\mathbf{K}_h$ see \eqref{qa5}, and assumption \eqref{53}, we easily get
  			\begin{align}\label{34}
  			\|\mathbf{w}_h\|_{1,h} \leq \frac{1}{\hat{\alpha}} \|\mathbf{f}\|_{\mathrm{V}^{\prime}} \leq \frac{\hat{\alpha}}{2},
  			\end{align}
  			and then, combining \eqref{33} and \eqref{34}, we obtain
  			\begin{align}\label{35}
  			\sup _{\mathbf{0} \neq  (\mathbf{v}_h,p_h) \in \mathrm{V}_h \times \Pi_h} \frac{\mathcal{A}_{\mathbf{w}_h}\left[(\mathbf{u}_h,p_h);(\mathbf{v}_h,q_h)\right]}{\|(\mathbf{v}_h,q_h)\|} \geq \frac{\hat{\alpha}}{2}\|(\mathbf{u}_h,p_h)\| \quad \forall(\mathbf{u}_h,p_h) \in  \mathrm{V}_h \times \Pi_h.
  			\end{align}
  			On the other hand, for a given $ (\mathbf{u}_h,q_h) \in \mathrm{V}_h \times \Pi_h$, we observe that
  			$$
  			\begin{aligned}
  			\sup_{\mathbf{0} \neq (\mathbf{v}_h,q_h) \in \mathbf{V}_h \times \Pi_h} \mathcal{A}_{\mathrm{w}_h}\left[(\mathbf{v}_h,q_h);(\mathbf{u}_h,p_h)\right] & \geq \sup_{\mathbf{0} \neq(\mathbf{v}_h,q_h) \in \mathrm{V}_h \times \Pi_h} \frac{\mathcal{A}_{\mathbf{w}_h}\left[(\mathbf{v}_h,q_h);(\mathbf{u}_h,p_h)\right]}{\|(\mathbf{v}_h,q_h)\|} \\
  			& =\sup _{\mathbf{0} \neq(\mathbf{v}_h,q_h) \in \mathrm{V}_h \times \Pi_h} \frac{\mathcal{C}_h\left[(\mathbf{v}_h,q_h);(\mathbf{u}_h,q_h)\right]+\mathbf{c}(\mathbf{w}_h ; \mathbf{v}_h, \mathbf{u}_h)}{\|(\mathbf{v}_h,q_h)\|},
  			\end{aligned}
  			$$
  			from which,
  			$$
  			\begin{aligned}
  			\sup _{\mathbf{0} \neq (\mathbf{v}_h,q_h) \in \mathrm{V}_h \times \Pi_h} \mathcal{A}_{\mathbf{w}_h}\left[( \mathbf{v}_h,q_h);(\mathbf{u}_h,q_h)\right] & \geq \sup _{\mathbf{0} \neq(\mathbf{v}_h,q_h) \in \mathrm{V}_h \times \Pi_h}\frac{|\mathcal{C}_h\left[(\mathbf{v}_h,q_h);( \mathbf{u}_h,p_h)\right]+\mathbf{c}(\mathbf{w}_h ; \mathbf{v}_h, \mathbf{u}_h)|}{\|( \mathbf{v}_h,q_h)\|} \\
  			& \geq \sup _{\mathbf{0} \neq(\mathbf{v}_h,q_h) \in \mathrm{V}_h \times \Pi_h} \frac{|\mathcal{C}_h\left[(\mathbf{v}_h,q_h);(\mathbf{u}_h,p_h)\right]|}{\|( \mathbf{v}_h,q_h)\|}-\frac{|\mathbf{c}(\mathbf{w}_h ; \mathbf{v}_h, \mathbf{u}_h)|}{\|( \mathbf{v}_h,q_h)\|},
  			\end{aligned}
  			$$
  			for all $\mathbf{0} \neq (\mathbf{v}_h,q_h) \in \mathrm{V}_h \times \Pi_h $, which together with Theorem~\ref{qa4}, implies
  			\begin{align}\label{36}
  			\sup _{\mathbf{0} \neq (\mathbf{v}_h,q_h) \in \mathrm{V}_h \times \Pi_h} \mathcal{A}_{\mathbf{w}_h}\left[(\mathbf{v}_h,q_h);(\mathbf{u}_h,p_h)\right] \geq \sup _{\mathbf{0} \neq ( \mathbf{v}_h,q_h) \in \mathrm{V}_h \times \Pi_h} \frac{\mathcal{C}_h\left[( \mathbf{v}_h,q_h);(\mathbf{u}_h,p_h)\right]}{\|(\mathbf{v}_h,q_h)\|}-\|\mathbf{w}_h\|_{1,h}\|( \mathbf{u}_h,p_h)\|.
  			\end{align}
  			Therefore, using the fact that $\mathcal{C}_h\left[\cdot ;\, \cdot\right]$ is symmetric, from the inequality in Theorem~\ref{51} and \eqref{36} we obtain
  			$$
  			\sup _{\mathbf{0} \neq (\mathbf{v}_h,q_h) \in \mathrm{V}_h \times \Pi_h} \mathcal{A}_{\mathbf{w}_h}\left[( \mathbf{v}_h,q_h);(\mathbf{u}_h,p_h)\right] \geq \hat{\alpha} \|( \mathbf{u}_h,p_h)\|-\|\mathbf{w}_h\|_{1,h}\|( \mathbf{u}_h,p_h)\|,
  			$$
  			which combined with \eqref{34}, yields
  			\begin{align}\label{37}
  			\sup _{\mathbf{0} \neq (\mathbf{v}_h,q_h) \in \mathrm{V}_h \times \Pi_h} \mathcal{A}_{\mathbf{w}_h}\left[(\mathbf{v}_h,q_h);(\mathbf{u}_h,p_h)\right] \geq \frac{\hat{\alpha}}{2}\|(\mathbf{u}_h,p_h)\|>0 \quad \forall(\mathbf{u}_h,p_h) \in \mathrm{V}_h \times \Pi_h,(\mathbf{u}_h,p_h) \neq 0.
  			\end{align}
  			By examining \eqref{35} and \eqref{37}, we can deduce that $\mathcal{A}_{\mathbf{w}_h}\left( \cdot , \cdot\right)$ satisfies the conditions of the Banach Ne\v{c}as Babu\v{s}ka theorem  \cite[Theorem 2.6]{MR4242224}. This guarantees the existence of a unique solution $(\mathbf{u}_h,p_h) \in \mathrm{V}_h \times \Pi_h$ to \eqref{F}, or equivalently, the existence of a unique $\mathbf{u}_h \in \mathrm{V}_h $ such that $\mathcal{J}_h(\mathbf{w}_h)=\mathbf{u}_h$. Furthermore, from \eqref{35} and \eqref{32} we derive the following inequality:
  			$$
  			\|\mathbf{u}_h\|_{1,h} \leq\|(\mathbf{u}_h,p_h)\| \leq \frac{1}{\hat{\alpha}} \|\mathbf{f}\|_{\mathrm{V}^{\prime}}.
  			$$
  		This concludes the proof by showing that $\mathbf{u_h}$ $\in$ $\mathbf{K_h}$.	
  		\end{proof} 
  		\subsubsection{Well-posedness of the discrete problem}
  		The subsequent theorem establishes the well-posedness of Nitsche's scheme \eqref{F}.
  		\begin{theorem}\label{van}
  			Let $\mathbf{f} \in \mathrm{V}^{\prime}$  such that
  			\begin{align}\label{55}
  			\frac{2}{\hat{\alpha}^2}\|\mathbf{f}\|_{\mathrm{V}^{\prime}}\leq 1.
  			\end{align}
  			Then, there exists a unique $\left( \mathbf{u}_h,p_h\right) \in \mathrm{V}_h \times \mathrm{\Pi}_h$ solution to \eqref{E}. In addition, there exists $C>0$, independent of $h$, such that
  			\begin{align}\label{qa8}
  			\left\|\mathbf{u}_h\right\|_{1,h}+\left\|p_h\right\|_{0,\Omega} \leq C\|\mathbf{f}\|_{\mathrm{V}^{\prime}}.
  			\end{align}
  		\end{theorem}
  		\begin{proof}
  			
  			According to the relations given in \eqref{44}, our aim is to establish well-posedness of \eqref{E}. This can be accomplished by demostrating that $\mathcal{J}_h$ possessess a unique fixed point in $\mathbf{K}_h$ using Banach's fixed point theorem.
  			\newline
  			The validity of Assumption \eqref{55} as shown in Theorem~\ref{qa7}, ensures the well-definedness of $\mathcal{J}_h$.
  			Now, let $\mathbf{w}_{h1}$, $\mathbf{w}_{h2}$, $\mathbf{u}_{h1}$, $\mathbf{u}_{h2}$ $\in \mathbf{K}_h$, be such that $\mathbf{u}_{h1}=\mathcal{J}_h\left(\mathbf{w}_{h1}\right)$ and $\mathbf{u}_{h2}=\mathcal{J}_h\left(\mathbf{w}_{h2}\right)$. By employing the definition of $\mathcal{J}$ and \eqref{32}, we can conclude the existence of unique $p_{h1}, p_{h2} \in {L}^2(\Omega)$, satisfies the following equations:
  			$$
  			\mathcal{A}_{\mathbf{w}_{h1}}\left[\left(\mathbf{u}_{h1},p_{h1}\right);( \mathbf{v}_h,q_h)\right]=\mathcal{F}(\mathbf{v}_h), \quad \text{and} \quad \mathcal{A}_{\mathbf{w}_{h2}}\left[\left( \mathbf{u}_{h2},p_{h2} \right);( \mathbf{v}_h,q_h)\right]=\mathcal{F}(\mathbf{v}_h) \quad \forall (\mathbf{v}_h,q_h) \in \mathrm{V}_h \times \Pi_h.
  			$$
  			By adding and subtracting appropriate terms, we can derive the following:
  			\begin{align}\label{40}
  			\mathcal{A}_{\mathbf{w}_{h1}}\left[\left( \mathbf{u}_{h1}-\mathbf{u}_{h2}; p_{h1}-p_{h2}\right),( \mathbf{v}_h,q_h)\right]=-\mathbf{c}\left(\mathbf{w}_{h1}-\mathbf{w}_{h2} ;  \mathbf{u}_{h2}, \mathbf{v}_h \right) \quad \forall(\mathbf{v}_h,q_h) \in \mathrm{V}_h \times \Pi_h.
  			\end{align}
  			Given that $\mathbf{w}_{h1} \in \mathbf{K}_h$ and using \eqref{40}, \eqref{35}, and Theorem~\ref{qa4}, we can deduce:
  			$$
  			\begin{aligned}
  			\frac{\hat{\alpha}}{2}\left\|\mathbf{u}_{h1}-\mathbf{u}_{h2}\right\|_{1} & \leq \sup _{\mathbf{0} \neq(\mathbf{v}_h,p_h) \in \mathrm{V}_h \times \Pi_h} \frac{\mathcal{A}_{\mathbf{w}_{h1}}\left[\left( \mathbf{u}_{h1}-\mathbf{u}_{h2},p_{h1}-p_{h2}\right);( \mathbf{v}_h,q_h)\right]}{\|( \mathbf{v}_h,q_h)\|} \\
  			& =\sup _{\mathbf{0} \neq( \mathbf{v}_h,q_h) \in \mathrm{V}_h \times \Pi_h} \frac{-\mathbf{c}\left(\mathbf{w}_{h1}-\mathbf{w}_{h2} ; \mathbf{u}_{h2}, \mathbf{v}_h\right)}{\|(\mathbf{v}_h,q_h)\|} \\
  			& \leq \left\|\mathbf{w}_{h1}-\mathbf{w}_{h2}\right\|_{1,h}\left\|\mathbf{u}_{h2}\right\|_{1,h},
  			\end{aligned}
  			$$
  			which together with the fact that $\mathbf{u}_{h2} \in \mathbf{K}_h$, implies
  			$$
  			\left\|\mathbf{u}_{h1}-\mathbf{u}_{h2}\right\|_{1,h} \leq \frac{1}{\hat{\alpha}} \|\mathbf{f}\|_{\mathrm{V}^{\prime}} \left\|\mathbf{w}_{h1}-\mathbf{w}_{h2}\right\|_{1,h}
  			$$
  			$$
  			\left\|\mathbf{u}_{h1}-\mathbf{u}_{h2}\right\|_{1,h} \leq \frac{\hat{\alpha}}{2}  \left\|\mathbf{w}_{h1}-\mathbf{w}_{h2}\right\|_{1,h}.
  			$$
  			Assumption \eqref{55} directly implies that $\mathcal{J}_h$ is a contraction mapping.
  			Now, to establish the estimate \eqref{qa8}, let $\mathbf{u}_h \in \mathbf{K}_h$ be the unique fixed point of $\mathcal{J}_h$ and let $\mathbf{u}_h \in \mathrm{V}_h$ be the unique solution of \eqref{E} with $( \mathbf{u}_h,p_h) \in \mathrm{V}_h \times \Pi_h $. By the definition of $\mathbf{K}_h$, it is evident that $\mathbf{u}_h$ satisfies the following
  			$$
  			\|\mathbf{u}_h\|_{1,h} \leq {\hat{\alpha}}^{-1}\|\mathbf{f}\|_{\mathrm{V}^{\prime}}.
  			$$
  			Consequently, utilizing \eqref{35} on $\mathcal{A}_{\mathbf{u}_h}$, referring back to the definition of $\mathcal{A}_{\mathbf{u}_h}$ given in \eqref{31}, and using the fact that $( \mathbf{u}_h,p_h)$ satisfies \eqref{E}, we obtain
  			$$
  			\|p_h\|_{\textbf0,\Omega}\leq\|( \mathbf{u}_h,p_h)\| \leq \frac{2}{\hat{\alpha}} \sup _{\mathbf{0} \neq( \mathbf{v}_h,q_h) \in \mathrm{V}_h \times \Pi_h} \frac{\mathcal{A}_{\mathbf{u}_h}\left[(\mathbf{u}_h,p_h);(\mathbf{v}_h,q_h)\right]}{\|( \mathbf{v}_h,q_h)\|}=\frac{2}{\hat{\alpha}} \sup _{\mathbf{0} \neq( \mathbf{v}_h,q_h) \in \mathrm{V}_h \times \Pi_h} \frac{\mathcal{F}(\mathbf{v}_h)}{\|( \mathbf{v}_h,p_h)\|},
  			$$
  			Thus
  			$$
  			\|p_h\|_{0,\Omega} \leq \frac{2}{\alpha} \|\mathbf{f}\|_{\mathrm{V}^{\prime}}.
  			$$	
  		\end{proof}
 Let us denote $\mathrm{I}_h$ be the interpolator operator, Under usual assumptions, the following approximation property hold.
\begin{lemma}\label{error}
   Let there exists $C_1>0$ and $C_2>0$, independent of $h$, such that for each $\mathbf{u} \in \textbf{\textit{H}}^{l+1}(T)$ with $0 \leq l \leq k$, there holds
$$
\begin{gathered}
\left\|\mathbf{u}-\mathrm{I}_h(\mathbf{u})\right\|_{\textbf{\textit{L}}^2(T)} \leq C_1 \frac{h_T^{l+2}}{\rho_T}|\mathbf{u}|_{\textbf{\textit{H}}^{l+1}(T)} \leq \hat{C}_1 h_T^{l+1}|\mathbf{u}|_{\textbf{\textit{H}}^{l+1}(T)} \\
\left|\mathbf{u}-\mathrm{I}_h(\mathbf{u})\right|_{\textbf{\textit{H}}^1(T)} \leq C_2 \frac{h_T^{l+2}}{\rho_T^2}|\mathbf{u}|_{\textbf{\textit{H}}^{l+1}(T)} \leq \hat{C}_2 h_T^l|\mathbf{u}|_{\textbf{\textit{H}}^{l+1}(T)}
\end{gathered}
$$
where $h_T$ is the diameter of $T, \rho_T$ is the diameter of the largest sphere contained in $T$, and $k$ is the degree of the polynomial. 
\end{lemma} 
\begin{proof}
    See \cite{MR2373954}.
\end{proof}
  		\section{A priori error bounds}
    \label{sec4}
  		The objective of this section is to establish the convergence of Nitsche's scheme \eqref{E} and determine the rate of convergence. We begin by deriving the corresponding C\`{e}a's estimate.
  		\begin{theorem}\label{RR1}
  			Assume that
  			\begin{align}\label{57}
  			\frac{2}{\alpha \hat{\alpha}} \|\mathbf{f}\|_{\mathrm{V}^{\prime}} \leq \frac{1}{2},
  			\end{align}
  			with $\alpha$ and $\hat{\alpha}$ being the positive constants in \eqref{29} and Theorem~\ref{51}, respectively. Let $(\mathbf{u},p) \in$ $\mathrm{V} \times \Pi$ and $\left(\mathbf{u}_h,p_h \right) \in \mathrm{V}_h \times \mathrm{\Pi}_h$ be the unique solutions of problems \eqref{9} and \eqref{E}, respectively. Then there exists $C_{\text{cea}}>0$, independent of $h$, such that
  			\begin{align}\label{58}
  			\left\|\left( \mathbf{u}-\mathbf{u}_h, p - p_h \right)\right\| \leq C_{\text{cea}} \inf _{\mathbf{0} \neq  \left(\mathbf{v}_h,q_h \right) \in \mathrm{V}_h \times \mathrm{\Pi}_h}\left\|\left(\mathbf{u}-\mathbf{v}_h, p-q_h \right)\right\|.
  			\end{align}	
  		\end{theorem} 
  		\begin{proof}
  			In order to simplify the subsequent analysis, we define $\mathbf{e}_{\mathrm{u}} = \mathbf{u}-\mathbf{u}_h$ and  $\mathbf{e}_p = p - p_h$, and for any $\left( \mathbf{z}_h, \zeta_h \right) \in \mathrm{V}_h \times \mathrm{\Pi}_h $, we write
  			\begin{align}\label{59}
  			\mathbf{e}_{\mathbf{u}}=\xi_{\mathbf{u}}+\chi_{\mathbf{u}}=\left(\mathbf{u}-\mathbf{z}_h\right)+\left(\mathbf{z}_h-\mathbf{u}_h\right), \text{and} \quad  \mathbf{e}_p=\xi_p+\chi_p=\left(p-\zeta_h\right)+\left(\zeta_h-p_h\right).
  			\end{align}
  			By recalling the definition of the bilinear forms $\mathcal{C}$ and $\mathcal{C}_h$ in \eqref{27} and \eqref{aa}, respectively, and considering \eqref{9} and \eqref{E}, we can observe the validity of following identities
  			$$
  			\mathcal{C}_h\left[( \mathbf{u},p);(\mathbf{v},q)\right]+\mathbf{c}(\mathbf{u} ; \mathbf{u}, \mathbf{v})= \mathcal{F}(\mathbf{v}) \quad \forall(\mathbf{v},q) \in \mathrm{V} \times \Pi.
  			$$
  			and
  			$$
  			\mathcal{C}_h\left[\left(\mathbf{u}_h,p_h \right);\left(\mathbf{v}_h,q_h\right)\right]+\mathbf{c}\left(\mathbf{u}_h ; \mathbf{u}_h, \mathbf{v}_h\right)=\mathcal{F}\left(\mathbf{v}_h\right) \quad \forall\left( \mathbf{v}_h,q_h \right) \in \mathrm{V}_h \times \mathrm{\Pi}_h.
  			$$
  			Based on these observations, we can deduce the Galerkin orthogonality property
  			\begin{align}\label{60}
  			\mathcal{C}_h\left[(\mathbf{e}_\mathbf{u}, \mathbf{e}_p);( \mathbf{v}_h,q_h)\right]+ \left[\mathbf{c}\left(\mathbf{u} ; \mathbf{u}, \mathbf{v}_h \right)-\mathbf{c}\left(\mathbf{u}_h ; \mathbf{u}_h, \mathbf{v}_h \right)\right]=0  \quad \forall\left( \mathbf{v}_h,q_h \right) \in \mathrm{V}_h \times {\Pi}_h.
  			\end{align}
  			Subsequently, by utilizing the decompositions given in \eqref{59}, the definition of $\mathcal{A}_{\mathrm{w}}$ in \eqref{31} for discrete $\mathcal{A}_{\mathrm{w}_h}$, and the identity
   \begin{align}\label{111}
  \mathbf{c}\left(\mathbf{u} ; \mathbf{u}, \mathbf{v}_h \right)=\mathbf{c}\left(\mathbf{u}-\mathbf{u}_h ; \mathbf{u}, \mathbf{v}_h\right)+\mathbf{c}\left(\mathbf{u}_h ; \mathbf{u}, \mathbf{v}_h\right).
  \end{align}\
     Now, using  \eqref{31}, \eqref{111}, and \eqref{60}, we deduce that for all $\left( \mathbf{v}_h,q_h \right) \in \mathrm{V}_h \times \mathrm{\Pi}_h $, the following relationship holds
  			$$
  			\begin{aligned}
  			\mathcal{A}_{\mathbf{u}_{\mathrm{h}}}\left[\left( \chi_{\mathbf{u}},\chi_p\right);\left(\mathbf{v}_h,q_h \right)\right]&= 
  			 \mathcal{C}_h \left[\left( \chi_{\mathbf{u}};\chi_p \right),\left( \mathbf{v}_h,q_h\right)\right]+\mathbf{c}\left(\mathbf{u}_h;{\chi_\mathbf{u}}, \mathbf{v}_h\right), 
  			 \\& = - \mathcal{C}_h\left[\left( \xi_{\mathbf{u}},\xi_p \right);\left( \mathbf{v}_h,q_h\right)\right] +
  			 \mathcal{C}_h\left[\left( \mathbf{e}_{\mathbf{u}},\mathbf{e}_p \right);\left( \mathbf{v}_h,q_h\right)\right]+\mathbf{c}\left(\mathbf{u}_h;{\chi_\mathbf{u}}, \mathbf{v}_h\right),
      \\& =  - \mathcal{C}_h\left[\left( \xi_{\mathbf{u}},\xi_p \right);\left( \mathbf{v}_h,q_h\right)\right] -  \mathbf{c}\left(\mathbf{u}-\mathbf{u}_h ; \mathbf{u}, \mathbf{v}_h\right)-\mathbf{c}\left(\mathbf{u}_h ; \mathbf{u}, \mathbf{v}_h\right)+ \mathbf{c}\left(\mathbf{u}_h ; \mathbf{u}_h,\mathbf{v}_h\right)+\mathbf{c}\left(\mathbf{u}_h;{\chi_\mathbf{u}}, \mathbf{v}_h\right),
      \\& =  - \mathcal{C}_h\left[\left( \xi_{\mathbf{u}},\xi_p \right);\left( \mathbf{v}_h,q_h\right)\right] -  \mathbf{c}\left(\xi_{\mathbf{u}}+\chi_{\mathbf{u}} ; \mathbf{u}, \mathbf{v}_h\right)-\mathbf{c}\left(\mathbf{u}_h ; \xi_{\mathbf{u}}+\chi_{\mathbf{u}}, \mathbf{v}_h\right)+\mathbf{c}\left(\mathbf{u}_h;{\chi_\mathbf{u}}, \mathbf{v}_h\right),
  			  \\& = - \mathcal{C}_h\left[\left( \xi_{\mathbf{u}},\xi_p \right);\left( \mathbf{v}_h,q_h\right)\right] -  \mathbf{c}\left(\xi_{\mathbf{u}}; \mathbf{u}, \mathbf{v}_h \right)-\mathbf{c}\left(\chi_{\mathbf{u}} ; \mathbf{u},\mathbf{v}_h\right)-\mathbf{c}\left(\mathbf{u}_h;{\xi_\mathbf{u}}, \mathbf{v}_h\right),
  	     	\end{aligned}
  			$$
  			which together with the definition of $\mathcal{C}_h$ given in equation \eqref{aa}, implies
  			\begin{align}\label{61}
  			\mathcal{A}_{\mathbf{u}_{\mathrm{h}}}\left[\left( \chi_{\mathbf{u}},\chi_p\right);\left(\mathbf{v}_h, q_h\right)\right]=& \nonumber -\mathbf{A}_h\left(\xi_\mathbf{u}, \mathbf{v}_h\right)-\mathbf{B}_h\left(\xi_\mathbf{u}, q_h \right)-\mathbf{B}_h\left({\mathbf{v}_h},\xi_p\right)-\mathbf{c}\left(\xi_{\mathbf{u}} ; \mathbf{u}, \mathbf{v}_h\right) 
  			 -\mathbf{c}\left(\chi_{\mathbf{u}} ; \mathbf{u}, \mathbf{v}_h\right)
  			 \\ 
  			 & -\mathbf{c}\left(\mathbf{u}_h ; \xi_{\mathbf{u}}, \mathbf{v}_h\right), 
  			\end{align}
  			for all $\left( \mathbf{v}_h,q_h\right) \in \mathrm{V}_h \times \mathrm{\Pi}_h$.
  			Next, utilizing the discrete inf-sup condition \eqref{37} at the left hand side of \eqref{61}, and applying the continuity properties of $\mathbf{A}, \mathbf{B}$, and $\mathbf{c}$ stated in Theorem~\ref{qa4} to the right hand side of \eqref{61}, we can derive the following
  			\begin{align*}
  			\left\|\chi_\mathbf{u}\right\|_{1,h} +\left\|\chi_p\right\|_{0}  &\lesssim \frac{2}{\hat{\alpha} \| (\mathbf{v}_h,q_h)\|}\bigg(\|\xi_{\mathbf{u}}\|_{1,h} \|\mathbf{v}_h\|_{1,h} + \|\xi_{\mathbf{u}}\|_{1,h} \|q_h\|_{0} + \|\xi_{\mathbf{p}}\|_{0} \|\mathbf{v}_h\|_{1,h} + \|\xi_{\mathbf{u}}\|_{1,h} \|\mathbf{v}_h\|_{1,h} \| \mathbf{u} \|_{1} + \\&   \|\chi_{\mathbf{u}}\|_{1,h} \|\mathbf{v}_h\|_{1,h} \| \mathbf{u} \|_{1} +  \|\xi_{\mathbf{u}}\|_{1,h} \|\mathbf{v}_h\|_{1,h} \| \mathbf{u}_h \|_{1,h} \bigg), \\
  		  &\lesssim \frac{2 }{\hat{\alpha}}\left(\left\|\xi_p\right\|_{0}+\left( 2+\left\|\mathbf{u}_h\right\|_{1,h}+\|\mathbf{u}\|_{1}\right)\left\|\xi_{\mathbf{u}}\right\|_{1,h}+\left\|\chi_{\mathbf{u}}\right\|_{1,h}\|\mathbf{u}\|_{1}\right),\\
  		 &\lesssim  \frac{2 }{\hat{\alpha}}\left(\left\|\xi_p\right\|_{0}+\left( 2+\left\|\mathbf{u}_h\right\|_{1,h}+\|\mathbf{u}\|_{1}\right)\left\|\xi_{\mathbf{u}}\right\|_{1,h}+\left\|\chi_{\mathbf{u}}\right\|_{1,h}\|\mathbf{u}\|_{1}\right),
  	\end{align*}
  			\begin{align}\label{62}
  			 \left\|\chi_p\right\|_{0}+\left(1-\frac{2}{\hat{\alpha}}\|\mathbf{u}\|_{1}\right)\left\|\chi_{\mathbf{u}}\right\|_{1,h}  &\lesssim \frac{2}{\hat{\alpha}}\left(\left\|\xi_p\right\|_{0}+\left\{2+\left\|\mathbf{u}_h\right\|_{1,h}+\|\mathbf{u}\|_{1}\right\}
  			 \left\|{\xi}_{\mathbf{u}}\right\|_{1,h}\right).
  			\end{align}
  			Therefore, by taking into account the fact that $\mathbf{u} \in \mathbf{K}$ and $\mathbf{u}_h \in \mathbf{K}_h$ based respectively on assumptions \eqref{57} and \eqref{62}, we can conclude that
  			\begin{align}\label{63}
  			\left\|\chi_p\right\|_{0}+\left\|\chi_{\mathbf{u}}\right\|_{1,h} \lesssim \left(\left\|\xi_p\right\|_{0}+\left\|\xi_{\mathbf{u}}\right\|_{1,h}\right).
  			\end{align}
            In this way, from \eqref{59}, \eqref{63} and the triangle inequality we obtain
  			$$
  			\left\|\left(\mathbf{e}_p, \mathbf{e}_{\mathbf{u}}\right)\right\| \leq\left\|\left(\chi_p, \chi_{\mathbf{u}}\right)\right\|+\left\|\left({\xi}_p, {\xi}_{\mathbf{u}}\right)\right\| \lesssim \left\|\left({\xi}_p, {\xi}_{\mathrm{u}}\right)\right\|.
  			$$
  		This, together with the fact that  $\left(\mathbf{z}_h,\zeta_h \right) \in \mathrm{V}_h \times \mathrm{\Pi}_h$ is arbitrary, leads to the conclusion of the proof.
  		\end{proof}
  		\begin{theorem}\label{101}
  			Let $(\mathbf{u},p) \in \mathrm{V} \times \Pi$ and $\left( \mathbf{u}_h,p_h\right) \in \mathrm{V}_h \times \mathrm{\Pi}_h$ denotes the unique solutions of the continuous problem \eqref{9} and discrete problem \eqref{E}, respectively, with $\mathbf{f}$   satisfying \eqref{57}. Suppose that
  			$(\mathbf{u},p) \in$ $  (\textbf{\textit{H}}^{l+1}(\Omega) \cap \mathrm{V}) \times ({H}^l(\Omega) \cap \Pi)$ with $l \geq 1$, Then there exists $C_{\text {rate }}>0$, independent of $h$, such that
  			$$
  			\left\|\left( \mathbf{u}-\mathbf{u}_h,p-p_h\right)\right\| \leq C_{\text {rate }} h^{l}\left\{|\mathbf{u}|_{\textbf{\textit{H}}^{l+1}(\Omega)}+|p|_{{H}^{l}(\Omega)}\right\}.
  			$$
  		\end{theorem}
  		\begin{proof}
  			The conclusion can be easily obtained by directly applying Theorem~\ref{RR1} and Lemma~\ref{error}.
  		\end{proof}
  		\section{Stabilized formulation for high Reynolds numbers}
    \label{sec5}
  The aim of this section is to provide a VMS-LES formulation of the Navier Stokes equations with slip boundary conditions. We validate with numerical tests the use of the Variational Multiscale (VMS) method with Nitsche in solving the Navier-Stokes equations in their standard weak form. In a time interval $(0,T]$ with $T>0$, the model problem reads:
  \begin{equation}\label{uns1}
  \begin{aligned}
  \frac{\partial \mathbf{u}}{\partial t} -\nu \Delta \mathbf{u} +  (\mathbf{u} \cdot \nabla) \mathbf{u}+\nabla p & =\mathbf{f} \quad \text {in}\, \Omega \times (0,T), \\
  \operatorname{div} \mathbf{u} & =0 \quad \text {in}\, \Omega \times (0,T), \\ 
  \mathbf{u} &= 0 \quad \text {on}\, \Gamma_D \times (0,T),\\
  \mathbf{u} \cdot \mathbf{n}&=0 \quad \text{on}\,  \Gamma_{\text{Nav}} \times (0,T), \\
  \nu \mathbf{n}^t D(\mathbf{u}) \boldsymbol{\tau}^k + \beta \mathbf{u} \cdot \boldsymbol{\tau}^k &=0\quad \text {on}\,  \Gamma_{\text{Nav}} \times (0,T), \quad k =1,2,\\
  \mathbf{u}(0) &= 0 \quad \text {in}\, \Omega \times \{0\}, \\
  (p, 1)_{\Omega} & =0.
  \end{aligned}
  \end{equation}
  The weak formulation of problem \eqref{uns1} can be written as 
  for all $ t \in (0,T]$,  Find $\left(\mathbf{u}, p \right) \in \mathrm{V} \times \Pi $ with $\mathbf{u}(0)=0$ such that
  \begin{align}\label{bb1}
  \mathcal{A}\left[\left(\mathbf{u}, p \right) ;\left(\mathbf {v}, q\right)\right]=\mathcal{F}(\mathbf{v})  \quad \forall  \in \left(\mathbf{v}, q \right) \in \mathrm{V} \times \Pi,
  \end{align}
  where
  $$
  \mathcal{A}\left[\left(\mathbf{u}, p \right);\left(\mathbf {v}, q\right)\right] \coloneqq \left( \frac{\partial \mathbf{u}}{\partial t},\mathbf{v}\right) + \frac{\nu}{2}\left({D}(\mathbf{u}), {D}(\mathbf{v})\right) +\left( \mathbf{u} \cdot \nabla \mathbf{u}, \mathbf{v}\right) -\left(p, \nabla \cdot \mathbf{v}\right) -\left(q, \nabla \cdot \mathbf{u}\right)+\int_{\Gamma_{\text{Nav}}} \beta \sum_i\left(\boldsymbol{\tau}^i \cdot \mathbf{v}\right)\left(\boldsymbol{\tau}^i \cdot \mathbf{u}\right) d s,
  $$
  $$
  \mathcal{F}(\mathbf{v}) \coloneqq \langle \mathbf{f},\mathbf{v}\rangle.
  $$
  \subsection{The VMS-LES formulation}
  The VMS technique involves decomposing the solution into coarser and finer scales. As a result, we decompose the weak formulation of the Navier-Stokes equations \eqref{bb1} into two subproblems, considering the coarse scale and the fine scale. The finite element method is used to approximate the coarse scale solution, while the fine scale solution is formulated analytically. Now, we decompose the
  		space into the direct sum of two subspaces:
  		\begin{align}\label{@1}
  		\mathcal{Y}_{0}=\mathcal{V}^h_0 \oplus \mathcal{V}^{\prime}_0
  		\end{align}
  where	$\mathcal{V}_0^h$ known as coarse scale spaces, are the finite element spaces used for the numerical discretization, i.e.
  	$$
  	\begin{aligned}
  	& \mathcal{V}_{0}^h=\mathrm{V}_h \times \Pi_h, 
  	\end{aligned}
  	$$
  	where $\mathcal{V}_0^{\prime}$ are infinite dimensional, known as the fine scale spaces, and orthogonal to $\mathcal{V}_0^h$ respectively. Then, we have the following decompositions from \eqref{@1}:
  		$$
  		\begin{aligned}
  		& \mathbf{u}=\mathbf{u}_h+\mathbf{u}^{\prime}, \\
  		& p=p_h+p^{\prime},
  		\end{aligned}
  		$$
  		where $(\mathbf{u},p) \in \mathcal{Y}_{0}$ and this is the starting point of VMS-LES method i.e.\, the separation of the flow field into resolved scales $(\mathbf{u}_h,p_h)$ and unresolved scales $(\mathbf{u}^{\prime},p^{\prime})$. 
  		Following the approach proposed in \cite{MR3920985}, 
  		Performing the two-scale decomposition on the problem \eqref{bb1}, we  obtain two subproblems, one for the coarse and one for the fine scales:
  		\begin{align}\label{1@}
  		\mathcal{A}\left[\mathbf{V}_h;\mathbf{U}_h+\mathbf{U}^{\prime}\right] = \mathcal{F}(\mathbf{V}_h)
  		\end{align}
  		\begin{align}\label{2@}
  		\mathcal{A}\left[\mathbf{V}^{\prime};\mathbf{U}_h+\mathbf{U}^{\prime}\right] = \mathcal{F}(\mathbf{V}^{\prime})
  		\end{align}
  	where the abbreviations $ \mathbf{U} = (\mathbf{u},p)$ and $\mathbf{V} = (\mathbf{v},q)$ are used for simplicity. It should be noted that the fine scale solution is typically modeled analytically,
expressed in terms of both the problem’s data and the coarse scale solution, and then substituted into the
coarse scale subproblem. By projecting the fine scale solution into the coarse scale solution, a finite dimensional
system for the coarse scale solution is obtained. 
  The solution of  \eqref{2@} is represented as
  	\begin{align}\label{uns3}
  	\mathbf{U}^{\prime}=F_{\mathbf{U}}(\mathbf{R}({\mathbf{U}_h})),
  	\end{align}
   which can be interpreted as the unresolved scales that are derived as a function of the residual of the resolved scales. By substituting \eqref{uns3} into the resolved scales equation \eqref{1@}, a unified set of equations for the resolved scales is obtained.
  	
The objective is to approximate $F_{\mathbf{U}}$ using models that are not dependent on the underlying physics of turbulent flows but are derived solely based on mathematical reasoning. Finally, we adopt a similar approach to \cite{MR1925888} in modeling the fine-scale velocity and pressure variables as:
  	$$
  	\begin{aligned}
  	& \mathbf{u}^{\prime} \simeq-\mathcal{S}_M\left(\mathbf{u}_h\right) \mathbf{r}_M\left(\mathbf{u}_h, p_h\right) \\
  	& p^{\prime} \simeq-\mathcal{S}_C\left(\mathbf{u}_h\right) r_C\left(\mathbf{u}_h\right)
  	\end{aligned}
  	$$
  	where $\mathbf{r}_M\left(\mathbf{u}_h, p_h\right)$ and $r_C\left(\mathbf{u}_h\right)$ indicate the strong residuals of the momentum and continuity equations:
  	$$
  	\begin{aligned}
  	& \mathbf{r}_M\left(\mathbf{u}_h, p_h\right)= \frac{\partial{\mathbf{u}}_h}{\partial t} + \mathbf{u}_h \cdot \nabla \mathbf{u}_h+\nabla p_h-\nu \Delta \mathbf{u}_h-\mathbf{f}, \\
  	& {r}_C\left(\mathbf{u}_h\right)=\nabla \cdot \mathbf{u}_h,
  	\end{aligned}
  	$$
  	respectively. Moreover, $\mathcal{S}_M$ and $\mathcal{S}_C$ are the stabilization parameters designed by a specific Fourier analysis applied in the framework of stabilized methods, which we choose similarly to \cite{MR2361475} as:
  	
  	\begin{align}\label{uns00}
  	\mathcal{S}_M\left(\mathbf{u}_h\right) & = \nonumber \left( \frac{\sigma^2}{\Delta{t}^2}+ \mathbf{u}_h \cdot \boldsymbol{G} \mathbf{u}_h+C_r \nu^2 \boldsymbol{G}: \boldsymbol{G}\right)^{-1 / 2} \\
  	\mathcal{S}_C\left(\mathbf{u}_h\right) & =\left(\mathcal{S}_M \boldsymbol{g} \cdot \boldsymbol{g}\right)^{-1}
  	\end{align}
  	where $\Delta t$ denotes the time step, while $\sigma$ denotes the order of the BDF (Backward Differentiation Formulas) time scheme \cite{MR1702282}. Furthermore, the constant $C_r = 60 \cdot 2^{r-2}$ is calculated using an inverse inequality which depends on the polynomial degree $r$ associated with the velocity finite element space \cite{MR3361036}. Moreover, $\boldsymbol{G}$ and $\boldsymbol{g}$ corresponds to the metric tensor and vector, respectively, and their definitions are defined as
  	$$
  	\begin{aligned}
  	& G_{i j}=\sum_{k=1}^d \frac{\partial \xi_k}{\partial x_i} \frac{\partial \xi_k}{\partial x_j}, \\
  	& g_i=\sum_{k=1}^d \frac{\partial \xi_k}{\partial x_i},
  	\end{aligned}
  	$$
 with $\boldsymbol{x}=\left\{x_i\right\}_{i=1}^d$ represents the coordinates of element $K$ in physical space, $\boldsymbol{\xi}=\left\{\xi_i\right\}_{i=1}^d$  represents the coordinates of element $\hat{K}$ in parametric space, and $\frac{\partial \boldsymbol{\xi}}{\partial \boldsymbol{x}}$ represents the inverse Jacobian of the element mapping between the reference and physical domains. We make the same assumptions as \cite{MR2361475}:
  	\begin{align}\label{uns11}
  	\begin{cases}
  	 &	 \frac{\partial \mathbf{v}_h}{\partial t}=0, \\
     & \mathbf{u}^{\prime}=0 \,\text{on}\, \Gamma,\\
     &  \left(D (\mathbf{v}_h), D( \mathbf{u}^{\prime})\right)=0.
  	\end{cases}
  	\end{align} 
 By explicitly expressing the left-hand side of \eqref{1@} and adding the Nitsche terms to the action of the Laplace distribution, we arrive at the following result:
  		$$
  		\begin{aligned}
  		\mathcal{A}\left[\mathbf{V}_h;\mathbf{U}_h+\mathbf{U}^{\prime}\right]=& ( \mathbf{u}_t,\mathbf{v}_h)+ \frac{\nu}{2}(D(\mathbf{u}_h), D(\mathbf{v}_h))+((\mathbf{u}_h+\mathbf{u^{\prime}}) \cdot \nabla (\mathbf{u}_h +\mathbf{u^{\prime}}), \mathbf{v}_h)-(p_h +p^{\prime}, \nabla \cdot \mathbf{v}_h)  \\& -(q_h, \nabla \cdot (\mathbf{u}_h + \mathbf{u^{\prime}})) + \int_{\Gamma_{\text{Nav}}} p^{\prime} (\mathbf{n} \cdot \mathbf{v}_h)  d s + \int_{\Gamma_{\text{Nav}}} q (\mathbf{n} \cdot \mathbf{u}^{\prime}) d s\\&+
  		\sum_{E\in \mathcal{E}_{\text{Nav}}}\bigg(-\int_{E} \mathbf{n}^t(\nu D(\mathbf{u}_h)-p_h I) \mathbf{n}(\mathbf{n} \cdot \mathbf{v}_h) d s-\int_{E} \mathbf{n}^t(\nu {D}(\mathbf{v}_h)-q_h I) \mathbf{n}(\mathbf{n} \cdot \mathbf{u}_h) d s \\& +\int_{E} \beta \sum_i\left(\boldsymbol{\tau}^i \cdot \mathbf{v}_h\right)\left(\boldsymbol{\tau}^i \cdot \mathbf{u}_h\right) d s 
  		+\gamma \int_{E} {h_e}^{-1}(\mathbf{u}_h \cdot \mathbf{n})(\mathbf{v}_h \cdot \mathbf{n}) d s \bigg), \\
  		=& ( \mathbf{u}_t,\mathbf{v}_h) +   \frac{\nu}{2}(D(\mathbf{u}_h), D(\mathbf{v}_h)) + ( \mathbf{u}_h \cdot \nabla \mathbf{u}_h, \mathbf{v}_h) - (p_h , \nabla \cdot \mathbf{v}_h) - (q_h, \nabla \cdot \mathbf{u}_h) \\& +
  			\sum_{E\in \mathcal{E}_{\text{Nav}}}\bigg(-\int_{E} \mathbf{n}^t(\nu D(\mathbf{u}_h)-p_h I) \mathbf{n}(\mathbf{n} \cdot \mathbf{v}_h) d s-\int_{E} \mathbf{n}^t(\nu {D}(\mathbf{v}_h)-q_h I) \mathbf{n}(\mathbf{n} \cdot \mathbf{u}_h) d s \\& +\int_{E} \beta \sum_i\left(\boldsymbol{\tau}^i \cdot \mathbf{v}_h\right)\left(\boldsymbol{\tau}^i \cdot \mathbf{u}_h\right) d s 
  			+\gamma \int_{E} {h_e}^{-1}(\mathbf{u}_h \cdot \mathbf{n})(\mathbf{v}_h \cdot \mathbf{n}) d s \bigg)  -(p^{\prime}, \nabla \cdot \mathbf{v}_h) \\& - (q_h, \nabla \cdot  \mathbf{u^{\prime}})  + \int_{\Gamma_{\text{Nav}}} p^{\prime} (\mathbf{n} \cdot \mathbf{v}_h)  d s + \int_{\Gamma_{\text{Nav}}} q (\mathbf{n} \cdot \mathbf{u}^{\prime}) d s  +( \mathbf{u}_h \cdot \nabla \mathbf{u^{\prime}}, \mathbf{v}_h) \\&+ ( \mathbf{u^{\prime}} \cdot \nabla \mathbf{u}_h , \mathbf{v}_h) + ( \mathbf{u^{\prime}} \cdot \nabla \mathbf{u^{\prime}}, \mathbf{v}_h).
  			\end{aligned}
  			$$	
Thereafter, we apply integration by parts to the fine-scale terms that appear in the coarse-scale equations, by considering the aforementioned assumption \eqref{uns11}. This results into the semi-discrete VMS-LES formulation of the Navier-Stokes equation with Nitsche which is expressed in terms of the weak residual as follows:
	for all $t \in (0,T]$, Find $\mathbf{U}_h=\left\{\mathbf{u}_h, p_h\right\} \in \mathcal{V}_0^h$ with $\mathbf{u}_h(0) = 0$ such that 
	\begin{align}\label{@10}
    \mathbf{H}\left[\mathbf{V}_h;\mathbf{U}_h\right]=\mathbf{L}\left(\mathbf{V}_h\right)
	\end{align}
for all $\mathbf{V}_h=\left\{\mathbf{v}_h, q_h\right\} \in \mathcal{V}_0^h$, where we considered the following definitions:
	$$
	\begin{aligned}
	& \mathbf{H}\left[\mathbf{V}_{{h}}; \mathbf{U}_{{h}}\right] \coloneqq {\mathcal{G}}^{N S}\left(\mathbf{V}_{{h}}, \mathbf{U}_{{h}}\right) + {\mathcal{G}}^{\text {SUPG }}\left(\mathbf{V}_{{h}}, \mathbf{U}_{{h}}\right)+{\mathcal{G}}^{\text {VMS }}\left(\mathbf{V}_{{h}}, \mathbf{U}_{{h}}\right)+ {\mathcal{G}}^{\text {LES }}\left(\mathbf{V}_{{h}}, \mathbf{U}_{{h}}\right), \\
	& \mathbf{L}\left(\mathbf{V}_{{h}}\right) \coloneqq \langle \mathbf{v}_h, \boldsymbol{f}\rangle,
	\end{aligned}
	$$
	with
\begin{align}\label{@3}
	{\mathcal{G}}^{NS}\left(\mathbf{V}_h, \mathbf{U}_h\right)= & \nonumber 
	\sum_{T\in \mathcal{T}_{h}}\bigg( \left(\mathbf{u}_t,\mathbf{v}_h\right)+ \frac{\nu}{2}\left(D(\mathbf{v}_h), D(\mathbf{u}_h)\right)+\left(  \mathbf{v}_h,\mathbf{u}_h \cdot \nabla \mathbf{u}_h\right)-\left(\nabla \cdot \mathbf{v}_h,p_h\right)-\left(q_h, \nabla \cdot \mathbf{u}_h\right)\bigg) +\\& \nonumber
	\sum_{E\in \mathcal{E}_{\text{Nav}}}\bigg(-\int_{E} \mathbf{n}^t\left(\nu D(\mathbf{u}_h)-p_h I\right) \mathbf{n}\left(\mathbf{n} \cdot \mathbf{v}_h\right) d s-\int_{E} \mathbf{n}^t\left(\nu {D}(\mathbf{v}_h)-q_h I \right) \mathbf{n} \left(\mathbf{n} \cdot \mathbf{u}_h \right) d s \\&  +\int_{E} \beta \sum_i\left(\boldsymbol{\tau}^i \cdot \mathbf{v}_h\right)\left(\boldsymbol{\tau}^i \cdot \mathbf{u}_h\right) d s 
	+\gamma \int_{E} {h_e}^{-1}\left(\mathbf{u}_h \cdot \mathbf{n}\right) \left(\mathbf{v}_h \cdot \mathbf{n}\right) d s \bigg) \\
	\mathcal{G}^{\text {SUPG}}\left(\mathbf{V}_h, \mathbf{U}_h\right)= & \sum_{T\in \mathcal{T}_{h}}\bigg( \left( \mathbf{u}_h \cdot \nabla \mathbf{v}_h-\tilde{C}\nabla q_h, \mathcal{S}_M\left(\mathbf{u}_h\right) \mathbf{r}_M\left(\mathbf{u}_h, p_h\right)\right)
	-\left(\nabla \cdot \mathbf{v}_h, \mathcal{S}_C\left(\mathbf{u}_h\right) \mathbf{r}_C\left(\mathbf{u}_h\right)\right)  \bigg) \label{@4}\\
  		\mathcal{G}^{\text {VMS}}\left(\mathbf{V}_h, \mathbf{U}_h\right)= & \nonumber
  		   \sum_{T\in \mathcal{T}_{h}}\bigg(\left( \mathbf{u}_h \cdot (\nabla \mathbf{v}_h)^T, \mathcal{S}_M\left(\mathbf{u}_h\right) \mathbf{r}_M\left(\mathbf{u}_h, p_h\right)\right) - 
  		   \sum_{E\in \mathcal{E}_{\text{Nav}}}\bigg(\int_{E} \mathcal{S}_C\left(\mathbf{u}_h \right) \mathbf{r}_C\left(\mathbf{u}_h\right) \left(\mathbf{n} \cdot \mathbf{v}_h\right) d s \\& + \int_{E} q_h \left(\mathbf{n} \cdot  \mathcal{S}_M\left(\mathbf{u}_h\right) \mathbf{r}_M\left(\mathbf{u}_h, {p}_h\right)\right) d s\bigg) \label{@5}\\
  		\mathcal{G}^{\text {LES}}\left(\mathbf{V}_h, \mathbf{U}_h\right)= & -\sum_{T\in \mathcal{T}_{h}} \bigg(\left(  \nabla \mathbf{v}_h, \mathcal{S}_M\left(\mathbf{u}_h\right) \mathbf{r}_M \left(\mathbf{u}_h, {p}_h \right) \otimes \mathcal{S}_M\left(\mathbf{u}_h\right) \mathbf{r}_M\left(\mathbf{u}_h, {p}_h\right)\right) \bigg). \label{@6}
  		\end{align}
  	In our formulation, we introduce an additional constant $\tilde{C}$ in \eqref{@4}  to discuss two choices for finite elements. Specifically, we set $\tilde{C}=0$ when we make use of $\mathbb P_2-\mathbb P_1$ inf-sup stable finite elements. Conversely, when we make use of stabilized equal order finite elements for both the velocity and pressure variables, we set $\tilde{C}=1$, as stated in \cite{MR3920985}.
  We would like to highligt that \eqref{@3} represents the weak formulation of the Navier-Stokes equation with Nitsche. We conclude that \eqref{@4} represents the classical Streamline Upwind Petrov Galerkin (SUPG) stabilization terms and \eqref{@5} represents additional terms introduced by VMS (Variational Multiscale) method. Finally, \eqref{@6} represents the LES (Large Eddy Simulation) modeling of turbulence.
  \begin{remark}
 This formulation differs from the standard one because of the contribution of pressure terms on the boundary $\Gamma_{\text{Nav}}$.
  \end{remark}
  \begin{remark}\label{uns01} 
  It is observed that as the time step $\Delta t$ approaches zero, the stabilization parameters in  \eqref{uns00} behave as follows:
  $$
  \mathcal{S}_M \sim \Delta t \rightarrow 0 \quad \mathcal{S}_C \sim \frac{1}{\Delta t} \rightarrow \infty.
  $$
A similar behavior is also demonstrated in \cite{perdoncin2020numerical} i.e., the VMS-LES modelling may lose its effectiveness for small time steps. It is observed that as $\Delta t \rightarrow 0$ , the term associated with the LES modeling of turbulence 
$\left(\nabla \mathbf{v}_h, \mathcal{S}_M\left(\mathbf{u}_h\right) \mathbf{r}_M \left(\mathbf{u}_h, {p}_h \right) \otimes \mathcal{S}_M\left(\mathbf{u}_h\right) \mathbf{r}_M\left(\mathbf{u}_h, {p}_h\right)\right)$ becomes negligible, while the dominant term \\$\left(\nabla \cdot \mathbf{v}_h, \mathcal{S}_C\left(\mathbf{u}_h\right)  
\mathbf{r}_C\left(\mathbf{u}_h\right)\right)$ fails to effectively act as a turbulence model in the semi-discrete VMS-LES weak formulation of the Navier-Stokes equations \eqref{@10}. 
  \end{remark}
\subsection{Fully discrete VMS-LES formulation}\label{pp}
  We obtain a fully discrete VMS-LES weak formulation of the Navier-Stokes equations with Nitsche by discretizing time with the BDF scheme of order $\sigma$, and the nonlinear terms in the above formulation are handled using Newton-Gregory backward polynomials \cite{cellier2006continuous}. A detailed explanation of the fully discrete VMS-LES method is provided in \cite{perdoncin2020numerical}.\\ \\
Find $ \mathbf{U}_{{h}} =\left\{\mathbf{u}^{n+1}_h, p^{n+1}_h\right\} \in \mathcal{V}_{0}^h$ :
   \begin{align}\label{full}
   {\tilde{\mathcal{G}}}^{N S}\left(\mathbf{V}_{{h}}, \mathbf{U}_{{h}}\right) + {\mathcal{\tilde{G}}}^{\text {SUPG }}\left(\mathbf{V}_{{h}}, \mathbf{U}_{{h}}\right)+{\mathcal{\tilde{G}}}^{\text {VMS }}\left(\mathbf{V}_{{h}}, \mathbf{U}_{{h}}\right)+ {\mathcal{\tilde{G}}}^{\text {LES }}\left(\mathbf{V}_{{h}}, \mathbf{U}_{{h}}\right) = \langle\mathbf{v}_h, \mathbf{f}^{n+1}\rangle,
   \end{align}
for all $ \mathbf{V}_{{h}} = \left\{\mathbf{v}_h, q_h\right\} \in \mathcal{V}_0^h$  with $\mathbf{f}^{n+1}=\mathbf{f}(t^{n+1})$.\\ \\
The bilinear forms $\tilde{\mathcal{G}}^{N S}\left(\mathbf{V}_{{h}}, \mathbf{U}_{{h}}\right), {\mathcal{\tilde{G}}}^{\text {SUPG }}\left(\mathbf{V}_{{h}}, \mathbf{U}_{{h}}\right),{\mathcal{\tilde{G}}}^{\text {VMS }}\left(\mathbf{V}_{{h}}, \mathbf{U}_{{h}}\right), \text{ and } {\mathcal{\tilde{G}}}^{\text {LES }}\left(\mathbf{V}_{{h}}, \mathbf{U}_{{h}}\right)$ are defined in Appendix~\ref{AA}. These forms are modifications of the bilinear forms defined in \eqref{@10}.
  	\section{Numerical Experiments}
   \label{sec6}
  	Now, computational examples are presented to demonstrate the consistency of the numerical scheme. The open-source finite element library FEniCS \cite{alnaes2015fenics} is utilized to simulate all numerical computations. The theoretical results of Theorem~\ref{101} is numerically validated in the first example. The Nitsche method is validated in the second example by comparing  with the benchmark problem found in the literature \cite{MR4307023}. The VMS-LES approach with Nitsche at high Reynolds numbers is validated in the last two examples. A Lagrange multiplier is used to implement the average zero condition for the pressure approximation. 
 	\subsection{Test 1: Convergence rates}
 	In this numerical test, we compute the convergence rate of the Nitsche method 
 \eqref{101}, considering the square domain  $\Omega=(-1,1)^2$, and a sequence of uniformly refined meshes. We present numerical test based on the following exact solution 
  		$$
  		\begin{gathered}
  		\mathbf{u}\left(x_1, x_2\right)=\left(2x_2(1-x_1^2),-2x_1(1-x_2^2)\right)^T 
  		\end{gathered}
  		$$
  		$$
  		\begin{gathered}
  		p \left(x_1,x_2 \right)= \left(2x_1-1 \right) \left(2x_2-1 \right).
  		\end{gathered}
  		$$
  		The slip boundary condition is imposed on $x_2=-1$ and the essential boundary  condition is enforced on the rest of the boundary.
      We can observe that Table~\ref{Nits} present the approximation errors for pressure and velocity as well as the convergence rate, which are in good agreement with the theory. Table~\ref{Nits1} presents the error in $L_2$ norm on the slip condition on $\Gamma_{\text{Nav}}$ and it shows that the larger the Nitsche parameter $\gamma$ is, the smaller error on the slip  condition. This happen because the error of the actual equation increases, so there is a compromise between both things. We can observe this behaviour in Table~\ref{Nits}. Additionally, we see that the number of Newton iterations required to reach the prescribed tolerance of $10^{-7}$ is at most three. Figure~\ref{con_test} represents the computed velocity field. The results in Table~\ref{Nits} and Figure~\ref{con_test} were computed with $\beta = 10$, $\nu = 1$, and $\gamma = 10$.

\begin{table}
  			\caption{Test 1: Experimental errors, iteration count, Number of degree of freedom (D.O.F.), and convergence rates for the approximate solutions $\mathbf{u}_h$ and $p_h$. Values are displayed for the Taylor-Hood space $\mathbb P_2-\mathbb P_1$ with $\beta = 10$ and $\nu = 1$.}
  			\label{Nits}
  			\begin{tabular}{|r|r|r|c|c|c|c|c|c|c|c|}
  				\toprule
  $\gamma$&	 Mesh & D.O.F. & Newton Its & $\|p-p_h\|_{0}$ & rate & $\|\nabla(\mathbf{u}-\mathbf{u}_h)\|$ & rate & $ \|\mathbf{u} - \mathbf{u}_h \|_{0} $ & rate \\
  				\midrule
   & $8 \times 8$ & $659$ & $3$ & $ 5.20 \times 10^{-2}$ & $-$ & $ 1.12 \times 10^{-1}$ & $-$& $ 4.90 \times 10^{-3}$ & - \\
  &	$16 \times 16$ &$2467$ & $2$ & $1.27 \times 10^{-2}$ & $2.03$ & $2.23 \times 10^{-2}$ & $2.32$ &$ 4.90 \times 10^{-4}$ & $3.34$\\
  		1.0	&	$32 \times 32$ & $9539$ & $2$ & $ 3.13 \times 10^{-3}$ & $2.02$ & $ 4.55 \times 10^{-3}$ & $2.29$ & $ 5.00 \times 10^{-5}$ &$3.30$\\
  			&	$64 \times 64$ & $37507$ & $2$ & $7.78 \times 10^{-4}$ & $2.01$ & $ 1.23 \times 10^{-3}$ & $1.88$ & $7.00\times 10^{-6}$ &$2.92$\\
  			&    $128 \times 128$ & $148739$ & $2$ & $1.94 \times 10^{-4}$ &  $2.00$& $2.59 \times 10^{-4}$ & $2.25$  & $ 1.00 \times 10^{-6}$ & $3.23$\\
  				\bottomrule
       & $8 \times 8$ & $659$ & $3$ & $ 5.18 \times 10^{-2}$ & $-$ & $ 8.33 \times 10^{-2}$ & $-$& $ 3.47 \times 10^{-3}$ & - \\
  			&	$16 \times 16$ &$2467$ & $2$ & $1.27 \times 10^{-2}$ & $2.02$ & $1.81 \times 10^{-2}$ & $2.19$ &$ 3.82 \times 10^{-4}$ & $3.18$\\
  	10		&	$32 \times 32$ & $9539$ & $2$ & $ 3.14 \times 10^{-3}$ & $2.02$ & $ 4.24 \times 10^{-3}$ & $2.10$ & $ 4.50 \times 10^{-5}$ &$3.09$\\
  			&	$64 \times 64$ & $37507$ & $2$ & $7.78 \times 10^{-4}$ & $2.01$ & $ 1.03 \times 10^{-3}$ & $2.04$ & $5.00 \times 10^{-6}$ &$3.04$\\
  			&    $128 \times 128$ & $148739$ & $2$ & $1.94 \times 10^{-4}$ &  $2.00$& $2.53 \times 10^{-4}$ & $2.02$  & $ 1.00 \times 10^{-6}$ & $3.01$\\
      \bottomrule
       & $8 \times 8$ & $659$ & $3$ & $ 5.15 \times 10^{-2}$ & $-$ & $ 6.32 \times 10^{-2}$ & $-$& $ 2.74 \times 10^{-3}$ & - \\
  			&	$16 \times 16$ &$2467$ & $2$ & $1.27 \times 10^{-2}$ & $2.02$ & $1.59 \times 10^{-2}$ & $1.98$ &$ 3.42 \times 10^{-4}$ & $3.00$\\
  		100	&	$32 \times 32$ & $9539$ & $2$ & $ 3.13 \times 10^{-3}$ & $2.02$ & $ 3.90 \times 10^{-3}$ & $1.99$ & $ 4.30 \times 10^{-5}$ &$3.00$\\
  			&	$64 \times 64$ & $37507$ & $2$ & $7.78 \times 10^{-4}$ & $2.01$ & $ 1.00 \times 10^{-3}$ & $1.99$ & $5.00 \times 10^{-6}$ &$2.99$\\
  			&    $128 \times 128$ & $148739$ & $2$ & $1.94 \times 10^{-4}$ &  $2.00$& $2.50\times 10^{-4}$ & $1.99$  & $ 1.00 \times 10^{-6}$ & $2.99$\\
     \bottomrule
  			\end{tabular}
  		\end{table}

  		\begin{table}
  			\centering
  			\caption{Test 1: Computation of $\|\mathbf{u}_h \cdot \mathbf{n}\|_{0,\Gamma_{\text{Nav}}}$ for different values of $\gamma$}
  			\label{Nits1}
  			\begin{tabular}{|r|r|c|c|c|c|c|}
  				\toprule
  				Mesh & $\gamma = 0.01$ & $\gamma = 0.1$ & $\gamma = 1$ & $\gamma = 10$ & $\gamma = 100$    \\
  				\midrule
  $8 \times 8$ & $ 1.09 \times 10^{-2}$ & $ 1.10 \times 10^{-2}$ & $ 1.87 \times 10^{-2}$ & $ 1.31\times 10^{-2}$ & $ 5.77 \times 10^{-4}$  \\
  $16 \times 16$ &$ 1.32 \times 10^{-3}$ & $ 1.37 \times 10^{-3}$ & $ 2.23 \times 10^{-3}$ & $1.46 \times 10^{-3}$ &$ 6.40 \times 10^{-5}$\\
  $32 \times 32$ & $2.16 \times 10^{-4}$ & $7.35 \times 10^{-4}$ & $ 2.24 \times 10^{-4}$ &  $ 1.60 \times 10^{-4}$ & $7.00 \times 10^{-6}$ \\
  $64 \times 64$ & $2.20 \times 10^{-5}$ & $ 1.90 \times 10^{-5}$ &  $ 4.90 \times 10^{-5}$ & $ 1.80 \times 10^{-5}$ & $ 1.00 \times 10^{-6}$ \\
  $128 \times 128$ & $3.00 \times 10^{-6}$ & $ 3.00 \times 10^{-6}$ &  $ 2.00 \times 10^{-6}$ & $ 2.00 \times 10^{-6}$  & $1.00 \times 10^{-7}$ \\
  				\bottomrule
  			\end{tabular}
  		\end{table}
  		\begin{figure}[htbp] 
  			\centering 
  			\includegraphics[width=0.6\textwidth]{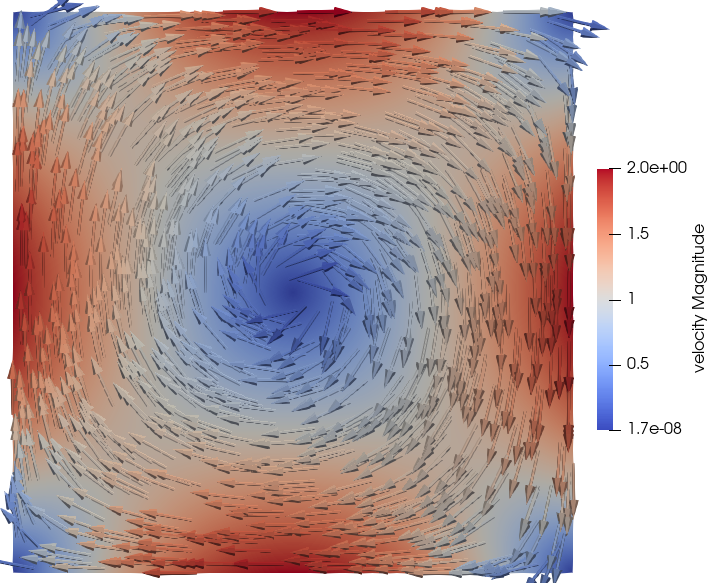}
  			\caption{Test 1: velocity field} 
  			\label{con_test} 
  		\end{figure}
  		
\subsection{Test 2: Lid-driven cavity test}
  This is a lid-driven cavity test and we perform it for steady and unsteady formulations. Consider the stationary Navier-Stokes equations with slip boundary conditions. The evaluation includes modelling a planar flow of an isothermal fluid inside a cavity driven. The cavity is represented as a square domain $\Omega=(0,1) \times(0,1)$, with negligible body force, and one moving wall. The velocity imposed on the top boundary $\{x_2=1\}$ is defined as
  		$$
  		\mathbf{u}= \begin{cases}(10 x_1, 0)^{\mathrm{T}} & \text {for}\, 0.0 \leqslant x_1 \leqslant 0.1 \\ (1,0)^{\mathrm{T}} & \text {for}\, 0.1 \leqslant x_1 \leqslant 0.9 \\ (10-10 x_1, 0)^{\mathrm{T}} & \text {for}\, 0.9 \leqslant x_1 \leqslant 1\end{cases}
  		$$
  		and the homogeneous slip boundary with $\beta = 1$ and $\gamma = 10$ is enforced on the other three sides. 
  	In Figure~\ref{v1-v2}, we observe the velocity streamlines for $Re = 1$ and $Re = 500$, confirming the expected behavior and aligning with the results reported in \cite{MR2683650}.\\
   
  Secondly, we consider the  unsteady Navier Stokes equations with the slip boundary condition and we apply the VMS-LES approach with Nitsche for validating the scheme at high Reynolds numbers. Figure~\ref{5000-15000-30000-50000} shows the velocity streamlines plots at Reynolds number $Re = 1000$ and $Re = 5000$ at final time $T = 35$ with time step $\Delta t = 0.035 s$. As the Reynolds number increases, the primary vortex migrates towards the center of the cavity or becomes more dense and the flow of the fluid is unpredictable in this region.

\begin{figure}
	\centering
 \begin{subfigure}[b]{0.4\textwidth}
		\centering
		\includegraphics[width=\textwidth]{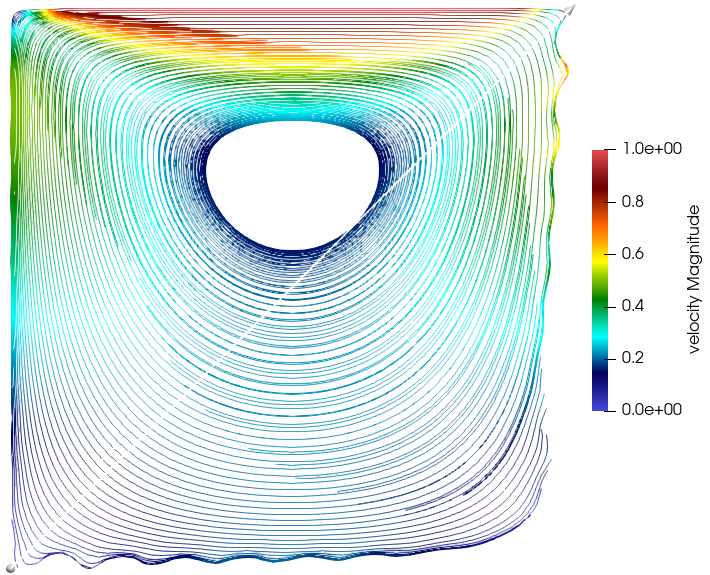}
        \caption{Re = 1}
	\end{subfigure}
	\hfill
	\begin{subfigure}[b]{0.4\textwidth}
		\centering
		\includegraphics[width=\textwidth]{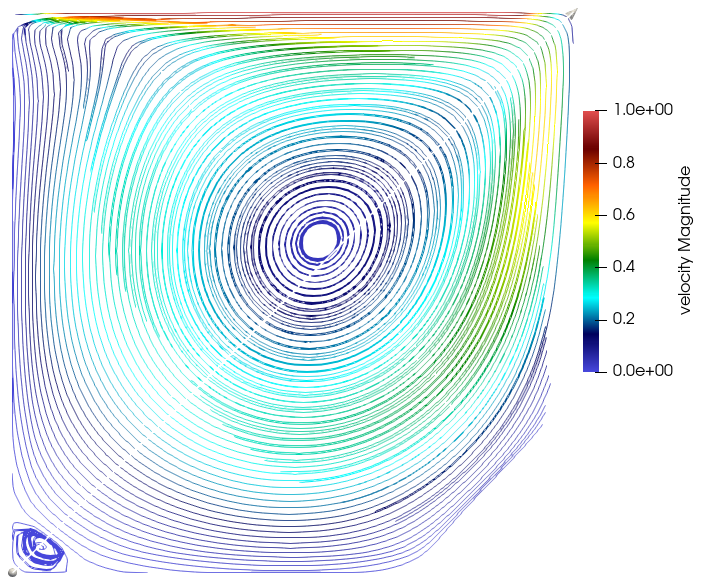}
  \caption{Re = 500}
	\end{subfigure}
      \caption{Test 2: Velocity streamlines $for Re=1$ and $Re=500 $ with mesh = 32 $\times$ 32}
	\label{v1-v2}
\end{figure}
\begin{figure}
	\centering
	\begin{subfigure}[b]{0.4\textwidth}
		\centering
		\includegraphics[width=\textwidth]{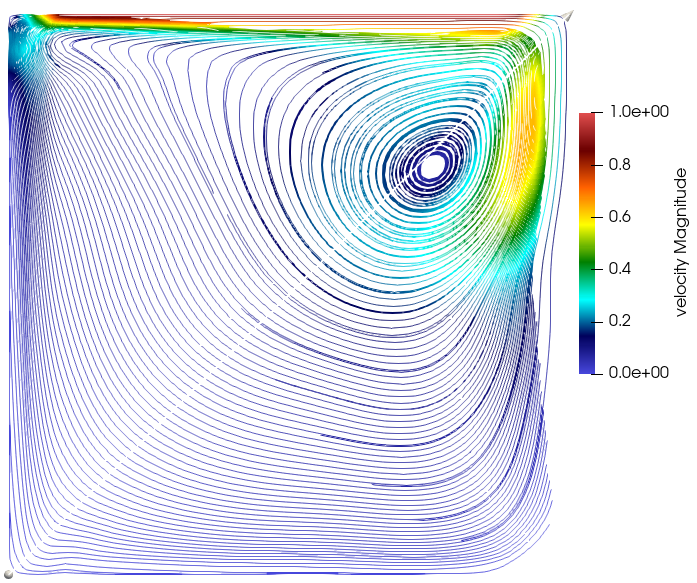}
		\caption{Re = 1000}
	\end{subfigure}
	\hfill
	\begin{subfigure}[b]{0.4\textwidth}
		\centering
		\includegraphics[width=\textwidth]{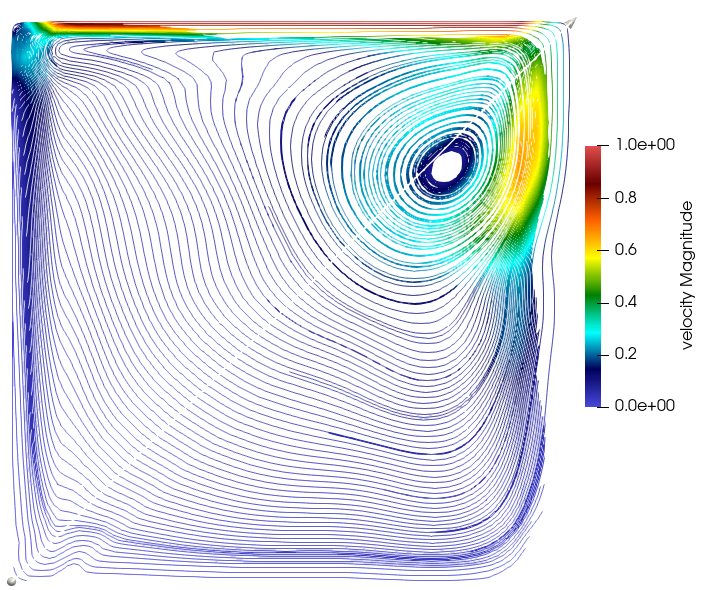}
		\caption{Re = 5000}
	\end{subfigure}
	\caption{Test 2: Velocity streamlines for $Re=1000$ and $Re=5000$  with mesh = 32 $\times$ 32 at $T = 35$ with $\Delta t=0.035s$}
	\label{5000-15000-30000-50000}
\end{figure}
\subsection{Test 3: Flow past through a circular cylinder}
This example is based on a standard three dimensional CFD benchmark problem: flow past through a circular cylinder. The geometrical settings of the domain is taken from \cite{araya:hal-04077986}. The computational mesh is depicted in Figure~\ref{cylinder}. 
\begin{figure}[htbp]
	\centering
	\includegraphics[width=0.8\textwidth]{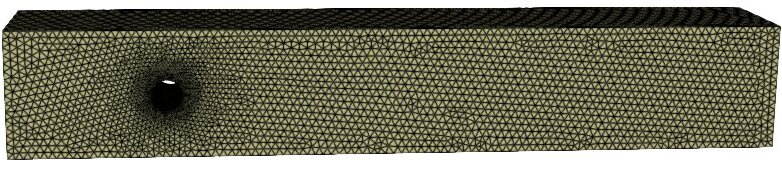}
	\caption {Test 3: Surface view of the computational mesh}
	\label{cylinder}
\end{figure}	
In this problem, no slip boundary condition are imposed on all the lateral walls of the box, while do-nothing boundary conditions is imposed at the outflow plane. On the surface of the cylinder we impose the homogeneous slip condition with $\beta = 1$ and $\gamma = 10$. Finally, the inflow condition is given by
$$
\mathbf{u}_D:=\left(\frac{16 U_m \sin(\pi t /8) x_2 x_3(H-x_2)(H-x_3)}{H^4}, 0,0\right)^T
$$
with $U_m:=2.25 \mathrm{~m} / \mathrm{s}$ and $H=0.41 \mathrm{~m}$. The Reynolds number is given by the formula $Re =\frac{UD }{\nu}$ where $U$ represents the average velocity of the fluid imposed on the inflow boundary, $D$ corresponds to the diameter of the cylinder. The mesh is depicted in Figure~\ref{cylinder}. Our objective is to represent the behavior of fluid velocity at high Reynolds numbers. The numerical solution of the streamlines of the fluid at different Reynolds numbers, including $1000$; $10\,000$; and $50\,000$ are observed at time $T=1$ with time step $\Delta t = 0.1 s$ in Figures~\ref{ap1},~\ref{ap2} and~\ref{ap3}. The isovalues of the pressure is presented at Reynold number $50,000$ in Figure~\ref{ap4}. The solution exhibits oscillations at larger time intervals, specifically, the tail of the flow after the obstacle develops typical oscillations.
\begin{figure}[htbp]
	\centering
	\includegraphics[width=0.8\textwidth]{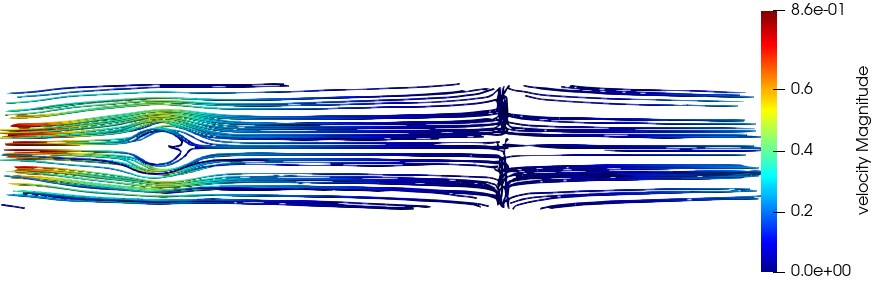}
	\caption {Test 3: Fluid velocity streamlines tubes at Re = 1000 at $T =1$ with $\Delta t = 0.1$}
	\label{ap1}
\end{figure}
\begin{figure}[htbp]
	\centering
	\includegraphics[width=0.8\textwidth]{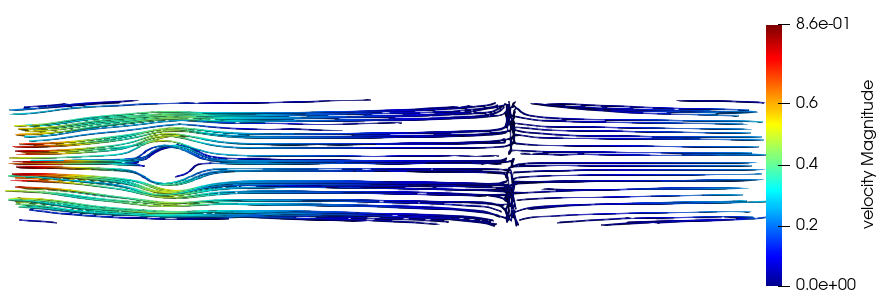}
	\caption {Test 3: Fluid velocity streamlines tubes at Re = 10\,000 at $T =1$ with $\Delta t = 0.1$}
	\label{ap2}
\end{figure}
\begin{figure}[htbp]
	\centering
	\includegraphics[width=0.8\textwidth]{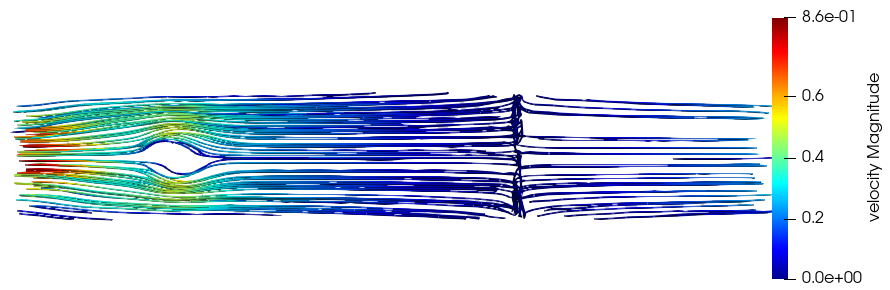}
	\caption {Test 3: Fluid velocity streamlines tubes at Re = 50\,000 at $T =1$ with $\Delta t = 0.1$}
	\label{ap3}
\end{figure}
\begin{figure}[htbp]
	\centering
	\includegraphics[width=0.8\textwidth]{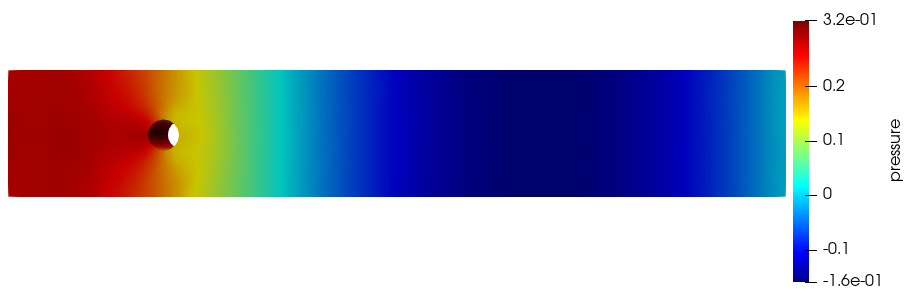}
	\caption {Test 3: Isovalues of the pressure at Re = 50\,000 at $T =1$ with $\Delta t = 0.1$}
	\label{ap4}
\end{figure}
\section{Conclusion}
In this paper, we address two main contributions. Firstly, we analyze Nitsche's method for the stationary Navier-Stokes equations on Lipschitz domains under minimal regularity assumptions. Our analysis provides a robust formulation for implementing slip (i.e., Navier) boundary conditions in arbitrarily complex boundaries. We establish the well-posedness of the discrete problem using the Banach Ne\v{c}as Babu\v{s}ka and the Banach fixed-point theorems under standard small data assumptions. Additionally, we provide optimal convergence rates for the approximation error. Secondly, we propose a Variational Multiscale Large Eddy Simulation (VMS-LES) stabilized formulation, which enables the simulation of incompressible fluids at high Reynolds numbers.
Finally, we perform three numerical tests: the first one validates the theoretical results of the Nitsche's scheme, the second one is a benchmark problem that demonstrates the consistency of our scheme for both steady and unsteady formulations at arbitrary Reynolds numbers, and the third test shows the behavior of the fluid passing through a cylinder at high Reynolds numbers.

 \section*{Acknowledgments}
AB was supported by the Ministry of Education, Government
of India - MHRD for financial assistance. NAB was supported by the ANID Grant \emph{FONDECYT de Postdoctorado N° 3230326}. 
\section*{Data Availability}
Enquiries about data availability should be directed to the authors.
\section*{Declarations}
\textbf{Conflict of interest} The authors have not disclosed any competing interests.

\bibliography{Navier-Stokes}  
\bibliographystyle{alpha}
\appendix
\section{Fully discrete VMS-LES formulation}\label{AA}
The bilinear forms appearing in the fully discrete VMS-LES formulation  (cf. Section~\ref{pp}) are defined below:
\begin{align*}
{\tilde{\mathcal{G}}}^{N S}\left(\mathbf{V}_{{h}}, \mathbf{U}_{{h}}\right) \coloneqq& \sum_{T\in \mathcal{T}_{h}}\bigg(\left(\mathbf{v}_h,  \frac{\alpha_\sigma \mathbf{u}^{n+1}_h-\mathbf{u}^{n, \mathrm{BDF}\sigma}_h}{\Delta t}\right) +\left( \mathbf{v}_h, \mathbf{u}^{n+1, \mathrm{EXT}}_h  \cdot \nabla \mathbf{u}^{n+1}_h\right)  + \frac{\nu}{2}\left(D (\mathbf{v}_h), D(\mathbf{u}_h^{n+1})\right) \\&-\left(\nabla \cdot \mathbf{v}_h, p^{n+1}_h\right)- \left(q_h, \nabla \cdot \mathbf{u}^{n+1}_h\right) \bigg) +
  	\sum_{E\in \mathcal{E}_{h}^b}\bigg(-\int_{E} \mathbf{n}^t(\nu D(\mathbf{u}_h^{n+1})-p_h^{n+1} I) \mathbf{n}(\mathbf{n} \cdot \mathbf{v}_h) d s \\& -\int_{E} \mathbf{n}^t(\nu {D}(\mathbf{v}_h)-q_h I) \mathbf{n}(\mathbf{n} \cdot \mathbf{u}_h^{n+1}) d s  +\int_{E} \beta \sum_i\left(\boldsymbol{\tau}^i \cdot \mathbf{v}_h\right)\left(\boldsymbol{\tau}^i \cdot \mathbf{u}_h^{n+1}\right) d s 
  	+\gamma \int_{E} {h_e}^{-1}(\mathbf{u}_h^{n+1} \cdot \mathbf{n})(\mathbf{v}_h \cdot \mathbf{n}) d s \bigg)
\end{align*}
\begin{align*}
 {\mathcal{\tilde{G}}}^{\text {SUPG }}\left(\mathbf{V}_{{h}}, \mathbf{U}_{{h}}\right) \coloneqq & \sum_{T\in \mathcal{T}_{h}}\bigg(\left(  \mathbf{u}^{n+1, \mathrm{EXT}}_h \cdot \nabla \mathbf{v}_h- \tilde{C}\nabla q_h, \mathcal{S}_M(\mathbf{u}^{n+1, \mathrm{EXT}}_h) \mathbf{r}_M(\mathbf{u}^{n+1}_h, p^{n+1}_h)\right) \\ & + \left(\nabla \cdot \mathbf{v}_h, \mathcal{S}_C\left(\mathbf{u}^{n+1, \mathrm{EXT}}_h\right) r_C(\mathbf{u}^{n+1}_h)\right) \\
{\mathcal{\tilde{G}}}^{\text {VMS }}\left(\mathbf{V}_{{h}}, \mathbf{U}_{{h}}\right) \coloneqq &	\sum_{T\in \mathcal{T}_{h}} \left( \mathbf{u}^{n+1, \mathrm{EXT}}_h \cdot\left(\nabla \mathbf{v}_h\right)^T, \mathcal{S}_M(\mathbf{u}^{n+1, \mathrm{EXT}}_h) \mathbf{r}_M(\mathbf{u}^{n+1}_h, p^{n+1}_h)\right)\\
-&\sum_{E\in \mathcal{E}_{h}^b}\bigg(\int_{E} \mathcal{S}_C(\mathbf{u}_h^{n+1, \mathrm{EXT}}) \mathbf{r}_C(\mathbf{u}_h^{n+1}) (\mathbf{n} \cdot \mathbf{v}_h) d s+ \int_{E} q_h (\mathbf{n} \cdot  \mathcal{S}_M\left(\mathbf{u}_h^{n+1, \mathrm{EXT}}\right) \mathbf{r}_M\left(\mathbf{u}_h^{n+1}, {p}_h^{n+1}\right)) d s\bigg) \\
{\mathcal{\tilde{G}}}^{\text {LES }}\left(\mathbf{V}_{{h}}, \mathbf{U}_{{h}}\right) \coloneqq & - \sum_{T\in \mathcal{T}_{h}} \left( \nabla \mathbf{v}_h, \mathcal{S}_M(\mathbf{u}^{n+1, \mathrm{EXT}}_h) \mathbf{r}_M(\mathbf{u}^{n+1, \mathrm{EXT}}_h, p^{n+1, \mathrm{EXT}}_h) \otimes \mathcal{S}_M\left(\mathbf{u}^{n+1, \mathrm{EXT}}_h\right) \mathbf{r}_M(\mathbf{u}^{n+1}_h, p^{n+1}_h)\right)
\end{align*}
for all $n \geq \sigma-1$,
  	given $\mathbf{u}^n_h, \ldots, \mathbf{u}^{n+1-\sigma}_h$, where :
  	$$
  	\begin{aligned}
  		\mathbf{u}^{n+1, \mathrm{EXT}}_h=& \begin{cases}\mathbf{u}^n_h & \text {for}\, \sigma=1, \text {if}\, n \geq 0 \\
  		2 \mathbf{u}^n_h-\mathbf{u}^{n-1}_h & \text {for}\, \sigma=2, \text {if}\, n \geq 1\end{cases} \\
  		p^{n+1, \mathrm{EXT}}_h=& \begin{cases}p^n_h & \text {for}\, \sigma=1, \text {if}\, n \geq 0 \\
  		2 p^n_h-p^{n-1}_h & \text {for}\, \sigma=2, \text {if}\, n \geq 1.\end{cases} \\
  		\frac{\partial \mathbf{u}_h}{\partial t} =& \begin{cases} \frac{\mathbf{u}^{n+1}_h - \mathbf{u}^{n}_h}{\Delta t} & \text {for}\, \sigma=1, \text {if}\, n \geq 1 \\
  	\frac{3 \mathbf{u}^{n+1}_h - 4 \mathbf{u}^n_h+ \mathbf{u}^{n-1}_h }{2 \Delta t} & \text {for}\, \sigma=2, \text {if}\, n \geq 2\end{cases} \\
  	\mathbf{u}^{n, \mathrm{BDF} \sigma}_h =& \begin{cases}\mathbf{u}^n_h & \text {for}\, \sigma=1, \text {if}\, n \geq 1 \\
 		2 \mathbf{u}^n_h-\frac{1}{2} \mathbf{u}^{n-1}_h & \text {for}\, \sigma=2, \text {if}\, n \geq 2 \end{cases} \\
   	\alpha_\sigma=& \begin{cases}1 & \text {for}\, \sigma=1 \\
  		\frac{3}{2} & \text {for}\, \sigma=2 \end{cases} \\
  	\mathbf{r}_M(\mathbf{u}, p)= &  \frac{\alpha_{\sigma} \mathbf{u}}{\Delta t}+ \mathbf{u}^{n+1, \mathrm{EXT}}_h \cdot \nabla \mathbf{u}+\nabla p-\nu \Delta \mathbf{u}- \frac{\mathbf{u}^{n, \mathrm{BDF \sigma}}_h}{\Delta t}-\mathbf{f}^{n+1}.
  \end{aligned}
  	$$
  	\begin{remark}
  	Here, we determine the values of the residual and stabilization parameters by substituting the values of $\mathbf{u}^{n+1, \mathrm{EXT}}_h$, $p^{n+1, \mathrm{EXT}}_h$, $\mathbf{u}^{n, \mathrm{BDF} \sigma}_h$ and $\alpha_\sigma$ that we require in our variational formulation.
  	\end{remark}
\end{document}